\documentclass[a4paper,12pt]{article}
\usepackage{amsthm}
\usepackage{amsmath,amsfonts,amssymb}
\usepackage{fourier}
\usepackage{a4wide}
\usepackage[usenames,dvipsnames]{color}
\usepackage[all,cmtip]{xy}
\usepackage[verbose,colorlinks=true,linktocpage=true,linkcolor=blue,citecolor=blue]{hyperref}
\usepackage{footnote}
\usepackage{tikz}
\usepackage{extarrows}
\usepackage[title]{appendix}

\theoremstyle{definition}
\newtheorem{definition}{Definition}[section]

\theoremstyle{plain}
\newtheorem{theorem}[definition]{Theorem}
\newtheorem{proposition}[definition]{Proposition}
\newtheorem{lemma}[definition]{Lemma}
\newtheorem{corollary}[definition]{Corollary}

\theoremstyle{remark}
\newtheorem{remark}[definition]{Remark}

\numberwithin{equation}{section}

\newcommand\keywords[1]{\textbf{keywords}:#1}

\allowdisplaybreaks[1]

\begin{document}

\title{From quantum loop superalgebras to super Yangians}
\author{Hongda Lin ${}^1$ Yongjie Wang ${}^2$ and Honglian Zhang ${}^{1,3}$}
\date{}
\maketitle

\begin{center}
\footnotesize
\begin{itemize}
\item[1] Department of Mathematics, Shanghai University, Shanghai 200444, China.
\item[2] School of Mathematics, Hefei University of Technology, Hefei, Anhui, 230009, China.
\item[3] Newtouch Center for Mathematics at Shanghai University, Shanghai 200444, China.
\end{itemize}
\end{center}

\begin{abstract}
The goal of this paper is to generalize a statement by Drinfeld, asserting that Yangians can be constructed as limit forms of the quantum loop algebras, to the super case. We establish a connection between quantum loop superalgebra and super Yangian of the general linear Lie superalgebra $\mathfrak{gl}_{M|N}$ in RTT type presentation. In particular, we derive the Poincar\'e-Birkhoff-Witt(PBW) theorem for the quantum loop superalgebra $\mathrm{U}_q\big(\mathfrak{Lgl}_{M|N}\big)$. Additionally, we investigate the application of the same argument to twisted super Yangian of the ortho-symplectic Lie superalgebra. For this purpose, we introduce the twisted quantum loop superalgebra as a one-sided coideal of $\mathrm{U}_q\big(\mathfrak{Lgl}_{M|2n}\big)$ with respect to the comultiplication.
\end{abstract}

\section{Introduction}
Quantum groups, first appearing in the theory of quantum integrable system, were formalized independently by Drinfeld and Jimbo as certain special Hopf algebras around 1985 \cite{D1,J}. These structures exhibited profound connections with various areas of mathematics and physics, including representation theory, algebraic combinatorics, and low dimensional
topology etc. Briefly, a quantum group can be understood as a $q$-deformation of the universal enveloping algebra of semi-simple Lie algebra or Kac-Moody algebra.
Among the significant families of quantized enveloping algebras, two hold particular importance: Yangians and quantum loop (affine) algebras.  Drinfeld in \cite{D2} proposed the Drinfeld type presentation of Yangians and quantum loop algebras, which were widely used in studying their highest weight representations and classifying of finite dimensional irreducible representations \cite{CP,HK}. Additionally, Yangians and quantum loop algebras possess the third presentations known as the RTT type presentations. This presentation, stemming from Faddeev, Reshetikhin and Takhtajan \cite{FRT}, provides a natural Hopf algebra structure that facilitates the investigation of tensor representations.

Sklyanin \cite{Sk} first introduced the twisted version of Yangians to study on boundary conditions for quantum integrable systems, named twisted Yangians by Olshanski \cite{Ol}. These algebras were constructed from classical symmetric pairs $\mathrm{AI}:\ (\mathfrak{gl}_N,\mathfrak{o}_N)$ and $\mathrm{AII}:\ (\mathfrak{gl}_N,\mathfrak{sp}_N)$. There are two methods to realize a twisted Yangian \cite{MNO}: One way is to define it as a subalgebra of the Yangian associated with $\mathfrak{gl}_N$ using the usual transposition, while the other is to construct an abstract algebra with generators subject to the quaternary relation and the symmetry relation. In \cite{GR}, the authors systematically studied twisted Yangians associated with the symmetric pairs of types $\mathrm{BDI}$, $\mathrm{CI}$, $\mathrm{CII}$, $\mathrm{DIII}$, which are coideal subalgebras of the Yangians for orthogonal or symplectic Lie algebras. More recently, Lu, Wang and Zhang \cite{LWZ} provide a Gauss decomposition approach to establish a Drinfeld type presentation for twisted Yangian of type $\mathrm{AI}$.

In 1986, Drinfeld stated that Yangians can be understood as limit forms of the quantum loop algebras in some sense, its proofs were provided by \cite{GM} and \cite{GT}. Conner and Guay further established a similar connection between the twisted quantum loop algebras and twisted Yangians of type $\mathrm{AI}$ and $\mathrm{AII}$ in \cite{CG} in terms of RTT type presentations.

Quantum supergroups, also referred to as quantum superalgebras, have emerged as supersymmetric generalizations of quantum groups \cite{BGZ,Y1}. The study of quantum affine superalgebras has been pursued by numerous researchers in related fields. Yamane \cite{Y2} systematically investigated the Serre-type presentations of type $\mathbf{A}$-$\mathbf{G}$ (quantum) affine Lie superalgebras and introduced the Drinfeld-Jimbo type presentations of quantum affine superalgebras. Yamane also constructed the Drinfeld type presentation of type $\mathbf{A}$ quantum affine Lie superalgebra. In a separate work, Y. -Z. Zhang in \cite{Zy} investigated the super RS construction and generalized the Ding-Frenkel \cite{DF} homomorphism to the quantum affine superalgebra associated with $\mathfrak{gl}_{M|N}$. Tsymbaliuk \cite{T21} employed the Shuffle algebra approach to provide a positive part of PBWD basis (i.e. Drinfeld type PBW basis) for type $\mathbf{A}$ quantum affine superalgebra. More recently, the authors in \cite{LYZ} constructed a minimalistic presentation quantum affine superalgebra of type $\mathbf{A}$ and demonstrated the isomorphism between Drinfeld-Jimbo type presentation and Drinfeld type presentation. These findings can be naturally generalized to quantum loop superalgebras, often perceived as quotient superalgebras of quantum affine superalgebras by the center.

In the early 1990s, Nazarov \cite{N} introduced a super-analogue $\mathrm{Y}\big(\mathfrak{gl}_{M|N}\big)$ by extending the definition of the Yangian $\mathrm{Y}\big(\mathfrak{gl}_{M}\big)$ in terms of RTT type presentation. Using the RTT type presentation and Gauss decomposition technique, Gow obtained the Drinfeld type presentation for $\mathrm{Y}\big(\mathfrak{gl}_{M|N}\big)$ with respect to distinguished Borel subalgebra \cite{G}. Likewise, Drinfeld type presentations for the super Yangians can be defined starting from non-conjugacy Dynkin diagrams, which are known to be pairwise isomorphic \cite{T}.
The finite dimensional irreducible representations of $\mathrm{Y}\big(\mathfrak{gl}_{M|N}\big)$ for the standard Dynkin diagram were classified by R. B. Zhang \cite{Zr} and for other Dynkin diagrams are investigated by the odd reflection established by \cite{Lu,M2}.
Y. -N. Peng established the parabolic presentation for super Yanigan of the general linear Lie superalgebra associated with arbitrary Borel subalgebras in \cite{P1,P2}.
Moreover, the definitions of the super Yangian and twisted super Yangian associated with orthosymplectic Lie superalgebras $\mathfrak{osp}_{M|N}$ were introduced in \cite{AACFR} and \cite{BR}, respectively. Recently, a detailed description of the finite dimensional irreducible representations of the orthosymplectic Yangian $\mathrm{Y}\big(\mathfrak{osp}_{M|N}\big)$ had been obtained by Molev in a series of papers \cite{M1,M3,M5}, and the Drinfeld type presentations were established in \cite{M4,MR}.

Inspired by the works of P. Conner, N. Guay and X. Ma \cite{CG,GM}, one naturally expects a direct connection between (twisted) super Yangians and (twisted) quantum loop superalgebras. In this paper, our primary target is to construct a filtered superalgebra with respect to a certain filtration of the quantum loop superalgebra,  which is isomorphic to the quantum super Yangian of type $\mathfrak{gl}_{M|N}$. Indeed, we use the fact that the $\mathbb{A}$-form of quantum loop superalgebra converges to the corresponding universal enveloping superalgebra in the classical limit. While this fact is widely accepted in the certified non-super version \cite{D1,Lus}, there are still not a systematic proof for super case. Hence, our initial focus in this paper is to study the $\mathbb{A}$-form. The second major goal is to extend a similar result to twisted case. To exploring twisted case, we introduce a new twisted quantum loop superalgebra that allows us to approach the twisted super Yangian described in \cite{BR}. As for orthosymplectic Yangians $\mathrm{Y}\big(\mathfrak{osp}_{M|N}\big)$, we will address them in an upcoming paper.

This paper is organized as follows. In Section 2, we begin by introducing some essential notations and reviewing the definitions of the super Yangian $\mathrm{Y}_{\hbar}\big(\mathfrak{gl}_{M|N}\big)$ and the quantum loop superalgebra $\mathrm{U}_q\big(\mathfrak{L}\mathfrak{gl}_{M|N}\big)$ exactly. We establish the PBW theorem for quantum loop superalgebra $\mathrm{U}_q\big(\mathfrak{L}\mathfrak{gl}_{M|N}\big)$ in a fixed order of the RTT type generators. Moreover, we introduce a RTT-version of $\mathbb{A}$-form, denoted as $\mathrm{U}_{\mathbb{A}}$, and establish a connection between $\mathrm{U}_{\mathbb{A}}$ and the universal enveloping superalgebra $\mathrm{U}\big(\mathfrak{L}\mathfrak{gl}_{M|N}\big)$. One of main results presented in this paper is the demonstration of the isomorphism between the super Yangian $\mathrm{Y}_{\hbar}\big(\mathfrak{gl}_{M|N}\big)$ and the graded algebra $\operatorname{gr}\mathrm{U}_{\mathbb{A}}$ subject to a suitable filtration. In Section 3, we recall the definition of the twisted super Yangian $\mathrm{Y}_{\hbar}^{tw}\big(\mathfrak{osp}_{M|2n}\big)$ under the super-transposition. Following that, we introduce the twisted quantum loop superalgebra as a sub-superalgebra of $\mathrm{U}_q\big(\mathfrak{L}\mathfrak{gl}_{M|2n}\big)$. Finally, we prove the statement that the twisted super Yangian $\mathrm{Y}_{\hbar}^{tw}\big(\mathfrak{osp}_{M|2n}\big)$ is isomorphic to a graded algebra based on the previously defined twisted quantum loop superalgebra.

\medskip

Recall the the complex set $\mathbb{C}$ and the integer set $\mathbb{Z}$. Throughout this paper, all superspaces, associative (super)algebras and Lie (super)algebras are over $\mathbb{C}$ if not specially stated. We write $\mathbb{Z}_2=\mathbb{Z}/2\mathbb{Z}=\left\{\bar{0},\bar{1}\right\}$ and the parity of an element $x$ is denoted by $|x|$. If both $\mathcal{A}$ and $\mathcal{B}$ are associative superalgebras, then $\mathcal{A}\otimes\mathcal{B}$ is understood as the associative superalgebra with graded multiplication
$$(a_1\otimes b_1)(a_2\otimes b_2)=(-1)^{|a_2||b_1|}a_1a_2\otimes b_1b_2, \text{ for homogeneous elements } a_2\in \mathcal{A}, b_1\in \mathcal{B}.$$

We use the following conventions in this paper. Let $\mathcal{I}$ be the ideal of an associative superalgebra $\mathcal{A}$. If $x\in\mathcal{A}$, $\overline{x}$ denotes the image of $a$ on the quotient $\mathcal{A}/\mathcal{I}$. The set $\mathbb{Z}_{\star}$ means that $\mathbb{Z}_{\star}\subseteq \mathbb{Z}$ and the elements contained in $\mathbb{Z}_{\star}$ satisfies the condition $\star$. The function $\delta_{\star}$ equals to 1 if the condition $\star$ is satisfied and 0 otherwise.

\section{Super Yangian and Quantum loop superalgebra}
In this section, we will review some basic definitions of the RTT type presentation of the Super Yangian $\mathrm{Y}\big(\mathfrak{gl}_{M|N}\big)$ and the quantum loop superalgebra $\mathrm{U}_q\big(\mathfrak{L}\mathfrak{gl}_{M|N}\big)$ to fix notations.

For nonnegative integers $M$ and $N$, we set $I_{M|N}:=\{1,\ldots,M+N\}$, simply denoted by $I$, on which we define the parity of $i\in I$ to be
$$|i|:=\begin{cases}\bar{0},&\text{if }\ 1\leqslant i\leqslant M,\\ \bar{1},&\text{if }\ M+1\leqslant i\leqslant M+N.\end{cases}$$
Let $V:=\mathbb{C}^{M|N}$ be the superspace with standard basis $v_i$ of parity $|i|$ for $i\in I$. Thus $\mathrm{End}(V)$ is an associative superalgebra with standard basis $E_{ij}$ of parity $|i|+|j|$ for $i,j\in I$. Under the standard supercommutator, $\mathrm{End}(V)$ is also a Lie superalgebra, which is denoted by $\mathfrak{gl}_{M|N}$.

We fix \textit{the standard Cartan sub-superalgebra} $\mathfrak{h}:=\text{Span}_{\mathbb{C}}\{h_i:=E_{ii}\ |\ i\in I\}$. Let $\{\epsilon_1,\ldots,\epsilon_{M+N}\}$ be the basis of $\mathfrak{h}^*$ dual to $\{h_1,\ldots, h_{M+N}\}$ with respect to the nondegenerate symmetric bilinear form $(\,\cdot\,|\,\cdot\,)$ given by $\left(\varepsilon_i\big|\varepsilon_j\right)=(-1)^{|i|}\delta_{ij}$. Then $P:=\mathbb{Z}\epsilon_1\oplus\ldots\oplus\mathbb{Z}\epsilon_{M+N}$ is \textit{the weight lattice}. For $i\in I_{M|N}^{\prime}=I\backslash\{M+N\}$ (simply denoted by $I^{\prime}$), we denote $\alpha_i=\varepsilon_i-\varepsilon_{i+1}$ to be the simple roots. Therefore the set $\Pi=\big\{\alpha_i\ |\ i\in I^{\prime}\big\}$ is  the root basis and the set $\Pi^{\vee}=\big\{\alpha_i^{\vee}\ |\ i\in I^{\prime}\big\}$ is the coroot basis with simple coroots $\alpha_i^{\vee}=h_i-(-1)^{|i|+|i+1|}h_{i+1}$. \textit{The special linear Lie superalgbera} $\mathfrak{sl}_{M|N}$ is a sub-superalgebra of $\mathfrak{gl}_{M|N}$ generated by $\alpha_i^{\vee}$ and $e_i=E_{i,i+1},\,f_i=E_{i+1,i}$ for $i\in I^{\prime}$. The Cartan matrix associated to standard Borel sub-superalgebra  of $\mathfrak{sl}_{M|N}$ is $A=\left(a_{ij}\right)_{i,j\in I^{\prime}}$ with $a_{ij}=\left(\alpha_i\big|\alpha_j\right)$, and the Dynkin diagram is given as follows,

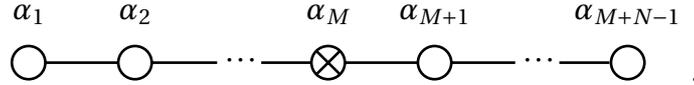
\begin{center}
	\begin{tikzpicture}
		\draw[line width =1pt] node at(1.62,0.64){$\alpha_1$} (1.62,0)circle[radius=0.22] (1.84,0)--(2.8,0);
		\draw[line width =1pt] node at(3.02,0.64){$\alpha_2$} (3.02,0)circle[radius=0.22] (3.24,0)--(4.1,0);
		\draw[line width =1pt]  node at(4.43,0){$\cdots$} (4.68,0)--(5.34,0);
		\draw[line width =1pt] (5.72,0.16)--(5.4,-0.16) (5.72,-0.16)--(5.4,0.16);
        \draw[line width =1pt] node at(5.56,0.64){$\alpha_M$} (5.56,0)circle[radius=0.22] (5.78,0)--(6.74,0);
        \draw[line width =1pt] node at(6.96,0.64){$\alpha_{M+1}$} (6.96,0)circle[radius=0.22] (7.18,0)--(8.02,0);
		\draw[line width =1pt]  node at(8.35,0){$\cdots$} (8.60,0)--(9.26,0);
		\draw[line width =1pt] node at(9.50,0.64){$\alpha_{M+N-1}$} (9.50,0)circle[radius=0.22];
	\end{tikzpicture}.
\end{center}

\subsection{Super Yangian $\mathrm{Y}\big(\mathfrak{gl}_{M|N}\big)$}
We recall the definition and some basic facts of the super Yangian $\mathrm{Y}\big(\mathfrak{gl}_{M|N}\big)$, more details can be found in \cite{G,N,T}. The \textit{super Yangian} $\mathrm{Y}\big(\mathfrak{gl}_{M|N}\big)$ is an associative superalgebra generated by the generators $\big\{t_{ij}^{(m)}|i,j\in I,m\in\mathbb{Z}_{>0} \big\}$ with the parity $|t_{ij}^{(m)}|=|i|+|j|$, convention $t_{ij}^{(0)}=\delta_{ij}$, and satisfing the following relation
\begin{gather}\label{Y1}
\left[t_{ij}^{(m)},\,t_{kl}^{(n)}\right]=(-1)^{|i||j|+|i||k|+|j||k|}\sum_{r=0}^{\min(m,n)-1}\left(t_{kj}^{(r)}t_{il}^{(m+n-1-r)}
-t_{kj}^{(m+n-1-r)}t_{il}^{(r)}\right).
\end{gather}
Equivalently,
\begin{gather}\label{Y2}
\left[t_{ij}^{(m+1)},\,t_{kl}^{(n)}\right]-\left[t_{ij}^{(m)},\,t_{kl}^{(n+1)}\right]
=(-1)^{|i||j|+|i||k|+|j||k|}\left(t_{kj}^{(m)}t_{il}^{(n)}-t_{kj}^{(n)}t_{il}^{(m)}\right),
\end{gather}
for all $i,j,k,l\in I$ and $n,m\geqslant 0$.

RTT type presentation of the super Yangian $\mathrm{Y}\big(\mathfrak{gl}_{M|N}\big)$ can be given as follows. The rational $R$-matrix $\mathsf{R}(u)=1\otimes 1-u^{-1}P\in \textrm{End}\left(V^{\otimes 2}\right)$, where $P=\sum\limits_{i,j\in I} (-1)^{|j|}E_{ij}\otimes E_{ji}$, satisfies the quantum Yang-Baxter equation
\begin{gather*}
\mathsf{R}^{12}(u)\mathsf{R}^{13}(u+v)\mathsf{R}^{23}(v)=\mathsf{R}^{23}(v)\mathsf{R}^{13}(u+v)\mathsf{R}^{12}(u),
\end{gather*}
where $\mathsf{R}^{ij}: V\otimes V\otimes V\longrightarrow V\otimes V\otimes V$ is a map acting on the $(i,j)$-th position by $\mathsf{R}$.
Let $$t_{ij}(u)=\sum\limits_{m\geq 0}t_{ij}^{(m)}u^{-m}\in\mathrm{Y}\big(\mathfrak{gl}_{M|N}\big)\left[\left[u^{-1}\right]\right],$$
and
$$\mathsf{T}(u)=\sum_{i,j\in I}(-1)^{|i||j|+|j|}E_{ij}\otimes t_{ij}(u).$$
Then the defining relation \eqref{Y1} can be rewritten as
\begin{gather}\label{Y3}
\mathsf{R}^{12}(u-v)\mathsf{T}^{1}(u)\mathsf{T}^{2}(v)=\mathsf{T}^{2}(v)\mathsf{T}^{1}(u)\mathsf{R}^{12}(u-v),
\end{gather}
where $\mathsf{T}^{1}(u)=\sum\limits_{i,j\in I }(-1)^{|i||j|+|j|}E_{ij}\otimes 1\otimes t_{ij}(u)$ and $\mathsf{T}^{2}(u)=\sum\limits_{i,j\in I }(-1)^{|i||j|+|j|}1\otimes E_{ij}\otimes t_{ij}(u)$ are in  $\text{End}(V^{\otimes 2} )\otimes \mathrm{Y}\big(\mathfrak{gl}_{M|N}\big)\left[\left[u^{-1}\right]\right].$

The super Yangian $\mathrm{Y}\big(\mathfrak{gl}_{M|N}\big)$ is a Hopf superalgebra endowed with the following comultiplication, and antipode given by
\begin{gather*}
\Delta_1\big(\mathsf{T}(u)\big)=\mathsf{T}(u)\otimes \mathsf{T}(u)\quad\textrm{and}\quad S_1\big(\mathsf{T}(u)\big)=\mathsf{T}(u)^{-1},\ \textrm{respectively}.
\end{gather*}

Let $\hbar\in\mathbb{C}\setminus\{0\}$ be the deformation parameter. We take a deformation of generators in $\mathrm{Y}\big(\mathfrak{gl}_{M|N}\big)$ by
$\hat{t}_{i,j}^{(m)}=\hbar^{1-m}t_{i,j}^{(m)}$ for all $i,j\in I, m\geqslant 1$ and change relation \eqref{Y2} to the following form
\begin{gather}\label{Y4}
\left[\hat{t}_{ij}^{(m+1)},\,\hat{t}_{kl}^{(n)}\right]-\left[\hat{t}_{ij}^{(m)},\,\hat{t}_{kl}^{(n+1)}\right]
=(-1)^{|i||j|+|i||k|+|j||k|}\hbar\left(\hat{t}_{kj}^{(m)}\hat{t}_{il}^{(n)}-\hat{t}_{kj}^{(n)}\hat{t}_{il}^{(m)}\right).
\end{gather}
The new superalgebra is called the \textit{quantum super Yangian}, denoted  by $\mathrm{Y}_{\hbar}\big(\mathfrak{gl}_{M|N}\big)$. The superalgebras $\mathrm{Y}_{\hbar}\big(\mathfrak{gl}_{M|N}\big)$ for all $\hbar\in\mathbb{C}\setminus\{0\}$ are pairwise isomorphic. It is a deformation of $\mathrm{U}\big(\mathfrak{gl}_{M|N}[x]\big)$ in terms of the limit $\hbar\mapsto 0$.

Take the lexicographical order "$\preceq$" on $I\times I\times \mathbb{Z}$ such that $(i_1,j_1;r_1)\preceq(i_2,j_2;r_2)$ for $i_1,i_2,j_1,j_2\in I$, $r_1,r_2\in\mathbb{Z}$ if and only if one of the following conditions hold:
\begin{gather*}
(1)\ i_1<i_2,\quad (2)\ i_1=i_2,\ j_1<j_2,\quad (3)\ i_1=i_2,\ j_1=j_2,\ r_1\leqslant r_2.
\end{gather*}
We call $(i_1,j_1;r_1)\prec(i_2,j_2;r_2)$ if $(i_1,j_1;r_1)\preceq(i_2,j_2;r_2)$ and $(i_1,j_1;r_1)\neq (i_2,j_2;r_2)$.
We say that the monomial
\begin{gather}\label{X:order:1}
x_{i_1,j_1}^{(r_1)}x_{i_2,j_2}^{(r_2)}\cdots x_{i_t,j_t}^{(r_t)}
\end{gather}
is an ordered monomial if $x_{i_1,j_1}^{(r_1)}\preceq\cdots\preceq x_{i_t,j_t}^{(r_t)}$.
Define an order on the set of the generators of $\mathrm{Y}_{\hbar}\big(\mathfrak{gl}_{M|N}\big)$ by
\begin{gather*}
\hat{t}_{i_1,j_1}^{(r_1)}\preceq \hat{t}_{i_2,j_2}^{(r_2)}\quad \textrm{if and only if}\quad (i_1,j_1;r_1)\preceq(i_2,j_2;r_2).
\end{gather*}
The next proposition is a natural generalization of the PBW theorem for super Yangian $\mathrm{Y}\big(\mathfrak{gl}_{M|N}\big)$ proved in \cite{G}.
\begin{proposition}\label{Y:ordered}
The set of all ordered monomials
\begin{gather}\label{Y:ordered:eq1}
\hbar^s\hat{t}_{i_1,j_1}^{(r_1)}\hat{t}_{i_2,j_2}^{(r_2)}\cdots \hat{t}_{i_t,j_t}^{(r_t)}
\end{gather}
with the restrictions
\begin{gather}\label{restriction1}
s,t\geqslant 0,\quad \textrm{and}~~\hat{t}_{i_a,j_a}^{(r_a)}\prec\hat{t}_{i_{a+1},j_{a+1}}^{(r_{a+1})}(1\leqslant a\leqslant t-1)~~\textrm{if}~~|i_a|+|j_a|=|i_{a+1}|+|j_{a+1}|=\bar{1}
\end{gather}
forms a basis of quantum super Yangian $\mathrm{Y}_{\hbar}(\mathfrak{gl}_{M|N})$.
\end{proposition}

\subsection{Quantum loop superalgebra $\mathrm{U}_q\big(\mathfrak{L}\mathfrak{gl}_{M|N}\big)$}
Let $\mathfrak{L}\mathfrak{gl}_{M|N}=\mathfrak{gl}_{M|N}\otimes\mathbb{C}[x,x^{-1}]$ be the loop superalgebra associated to the general linear Lie superalgebra.
It admits a Lie superalgebra structure with the standard supercommutator.
\begin{definition}\label{Uni}
The \textit{universal enveloping superalgebra} $\mathrm{U}\big(\mathfrak{L}\mathfrak{gl}_{M|N}\big)$ is an associative superalgebra generated by $E_{i,j}^{(r)}$ for $i,j\in I$, $r\in\mathbb{Z}$ with the parity $|E_{ij}^{(r)}|=|i|+|j|$, satisfying the following non-trivial relations
\begin{gather}\label{Uni:eq}
\left[E_{ij}^{(r)},\,E_{kl}^{(s)}\right]=\delta_{jk}E_{il}^{(r+s)}-\varepsilon_{ij;kl}\delta_{il}E_{kj}^{(r+s)},
\end{gather}
where we use the notation $\varepsilon_{ij;kl}=(-1)^{(|i|+|j|)(|k|+|l|)}$ for $i,j,k,l\in I$.

\end{definition}
It is widely recognized that there exists a canonical embedding map $\mathfrak{Lgl}_{M|N}\rightarrow \mathrm{U}\big(\mathfrak{Lgl}_{M|N}\big)$ such that $E_{ij}\otimes x^r\mapsto E_{ij}^{(r)}$. The element $E_{ij}\otimes f(x)$ for $f(x)\in \mathbb{C}[x,x^{-1}]$ under this map is often denoted by $E_{ij}f(x)$. In particular, $E_{ij}x^r=E_{ij}^{(r)}$.
Define an order on the set of the generators of $\mathrm{U}\big(\mathfrak{Lgl}_{M|N}\big)$ by
\begin{align*}
&E_{i_1,j_1}^{(r_1)}\preceq E_{i_2,j_2}^{(r_2)}\quad \textrm{if and only if}\quad (i_1,j_1;r_1)\preceq(i_2,j_2;r_2), \\
&E_{i_1,j_1}^{(-r_1)}\preceq E_{i_2,j_2}^{(-r_2)}\quad \textrm{if and only if }\quad (i_1,j_1;r_1)\preceq(i_2,j_2;r_2), \\
&E_{i,j}^{(-r)}\preceq E_{i,j}^{(r)}\quad \textrm{for any possible}\quad (i,j;r),
\end{align*}
where $r_1,r_2\geqslant 0$ and the order "$\preceq$" is the same as Section 2.1.
The next proposition provides a PBW type basis of $\mathrm{U}\big(\mathfrak{Lgl}_{M|N}\big)$, which directly follows from the discussion in \cite[Section 2.3]{Sc}.
\begin{proposition}\label{U:ordered}
The set of all ordered monomials \eqref{X:order:1} in the generators $E_{ij}^{(r)}$ with the restrictions
\begin{gather}\label{restriction2}
t\geqslant 0,\quad \textrm{and}~~E_{i_a,j_a}^{(r_a)}\prec E_{i_{a+1},j_{a+1}}^{(r_{a+1})}(1\leqslant a\leqslant t-1)~~\textrm{if}~~|i_a|+|j_a|=|i_{a+1}|+|j_{a+1}|=\bar{1}
\end{gather}
forms a basis of $\ \mathrm{U}\big(\mathfrak{Lgl}_{M|N}\big)$.
\end{proposition}

\vspace{1em}
\subsubsection{Definition of the quantum loop superalgebra $\mathrm{U}_q\big(\mathfrak{L}\mathfrak{gl}_{M|N}\big)$}
Let q be an indeterminate. To study the RTT type presentation of quantum affine superalgebra $\mathrm{U}_q\big(\mathfrak{L}\mathfrak{gl}_{M|N}\big)$, we first review the Perk-Schultz solution $\Re_q(u,v)\in \text{End}\left(V^{\otimes 2}\right)[u,v]$ of the quantum Yang-Baxter equation (c.f. \cite{PS,Zh})
\begin{gather*}
\Re_q(u,v)=\sum_{i,j\in I}\left(uq_i^{\delta_{ij}}-vq_i^{-\delta_{ij}}\right)E_{ii}\otimes E_{jj}+u\sum_{i>j}\left(q_j-q_j^{-1}\right)E_{ij}\otimes E_{ji}+v\sum_{i<j}\left(q_j-q_j^{-1}\right)E_{ij}\otimes E_{ji},
\end{gather*}
which satisfies the following quantum Yang-Baxter equation
\begin{gather*}
\Re_q^{12}(u,v)\Re_q^{13}(u,w)\Re_q^{23}(v,w)=\Re_q^{23}(v,w)\Re_q^{13}(u,w)\Re_q^{12}(u,v).
\end{gather*}
Here $q_i:=q^{(-1)^{|i|}}$ and $\Re_q^{ij}(1\leq i<j\leq 3)$ is obtained from $\Re_q$ similarly as in Section 2.1. Then the quantum loop superalgebra $\mathrm{U}_q\big(\mathfrak{L}\mathfrak{gl}_{M|N}\big)$ is defined via the Faddeev-Reshetikhin-Takhtajan's presentation as follows.

\begin{definition}
The \textit{quantum loop superalgebra} $\mathrm{U}_q\big(\mathfrak{L}\mathfrak{gl}_{M|N}\big)$ in terms of its RTT type presentation is an associative superalgebra generated by the set of generators $\big\{T_{ij}^{(r)},\widetilde{T}_{ij}^{(r)}\big|\ i,j\in I,r\in\mathbb{Z}_{\geqslant 0}\big\}$ with the parities $\big|T_{ij}^{(r)}\big|=\big|\widetilde{T}_{ij}^{(r)}\big|=|i|+|j|$ for all $r\in\mathbb{Z}_{\geqslant 0}$, satisfying the following relations
\begin{align}\label{RTT1}
&T_{ji}^{(0)}=\widetilde{T}_{ij}^{(0)}=0,~\textrm{if\ }~1\leqslant i<j\leqslant M+N, \\ \label{RTT2}
&T_{ii}^{(0)}\widetilde{T}_{ii}^{(0)}=\widetilde{T}_{ii}^{(0)}T_{ii}^{(0)}=1,~\textrm{if\ }~i\in I, \\ \label{RTT3}
&\Re_q^{12}(u,v)T^{1}(u)T^{2}(v)=T^{2}(v)T^{1}(u)\Re_q^{12}(u,v), \\ \label{RTT4}
&\Re_q^{12}(u,v)\widetilde{T}^{1}(u)\widetilde{T}^{2}(v)=\widetilde{T}^{2}(v)\widetilde{T}^{1}(u)\Re_q^{12}(u,v), \\ \label{RTT5}
&\Re_q^{12}(u,v)\widetilde{T}^{1}(u)T^{2}(v)=T^{2}(v)\widetilde{T}^{1}(u)\Re_q^{12}(u,v),
\end{align}
where the matrices $T(u)$ and $\widetilde{T}(u)$ have the forms
\begin{gather*}
T(u)=\sum_{i,j\in I}E_{ij}\otimes T_{ij}(u)\ \ \text{ with }\ T_{ij}(u)=\sum_{r\geqslant 0}T_{ij}^{(r)}u^{r},
\end{gather*}
and
\begin{gather*}
\widetilde{T}(u)=\sum_{i,j\in I}E_{ij}\otimes \widetilde{T}_{ij}(u)\ \ \text{ with }\ \widetilde{T}_{ij}(u)=\sum_{r\geqslant 0}\widetilde{T}_{ij}^{(r)}u^{-r}.
\end{gather*}
\end{definition}

It is a Hopf superalgebra endowed with the comultiplication
\begin{gather*}
\Delta_2:\ T_{ij}(u)\mapsto\sum_{k\in I}\varepsilon_{ik;kj}T_{ik}(u)\otimes T_{kj}(u),\quad \widetilde{T}_{ij}(u)\mapsto\sum_{k\in I}\varepsilon_{ik;kj}\widetilde{T}_{ik}(u)\otimes \widetilde{T}_{kj}(u),
\end{gather*}
and antipode $S_2:\ T(u)\mapsto T(u)^{-1},\ \widetilde{T}(u)\mapsto \widetilde{T}(u)^{-1}$, respectively.

We often use a more explicit expression in terms of the generators for relations \eqref{RTT3}$-$\eqref{RTT5}. More precisely, we can rewrite the defining relation \eqref{RTT5} in terms of the generators $T_{ij}^{(r)},\widetilde{T}_{ij}^{(r)}$ as
\begin{equation}\label{RTT5n5}
\begin{split}
&q_i^{-\delta_{ik}}\widetilde{T}_{ij}^{(r)}T_{kl}^{(s-1)}-q_i^{\delta_{ik}}\widetilde{T}_{ij}^{(r+1)}T_{kl}^{(s)}
-\varepsilon_{ij;kl}\left(q_j^{-\delta_{jl}}T_{kl}^{(s-1)}\widetilde{T}_{ij}^{(r)}-q_j^{\delta_{jl}}T_{kl}^{(s)}\widetilde{T}_{ij}^{(r+1)}\right) \\
=&\,\varepsilon_{ik;kl}\left(q_k-q_k^{-1}\right)\left(\delta_{k<i}\widetilde{T}_{kj}^{(r+1)}T_{il}^{(s)}
-\delta_{j<l}T_{kj}^{(s)}\widetilde{T}_{il}^{(r+1)}+\delta_{i<k}\widetilde{T}_{kj}^{(r)}T_{il}^{(s-1)}
-\delta_{l<j}T_{kj}^{(s-1)}\widetilde{T}_{il}^{(r)}\right).
\end{split}
\end{equation}
The defining relations in terms of the generators $T_{ij}^{(r)}$ is obtained from \eqref{RTT5n5} by replacing $\widetilde{T}_{ij}^{(r)}$ by $T_{ij}^{(r)}$,
\begin{equation}\label{RTT6}
\begin{split}
&q_i^{-\delta_{ik}}T_{ij}^{(r+1)}T_{kl}^{(s)}-q_i^{\delta_{ik}}T_{ij}^{(r)}T_{kl}^{(s+1)}
-\varepsilon_{ij;kl}\left(q_j^{-\delta_{jl}}T_{kl}^{(s)}T_{ij}^{(r+1)}-q_j^{\delta_{jl}}T_{kl}^{(s+1)}T_{ij}^{(r)}\right) \\
=&\,\varepsilon_{ik;kl}\left(q_k-q_k^{-1}\right)\left(\delta_{k<i}T_{kj}^{(r)}T_{il}^{(s+1)}-\delta_{j<l}T_{kj}^{(s+1)}T_{il}^{(r)}+\delta_{i<k}T_{kj}^{(r+1)}T_{il}^{(s)}
-\delta_{l<j}T_{kj}^{(s)}T_{il}^{(r+1)}\right),
\end{split}
\end{equation}
and the defining relations in terms of the generators $\widetilde{T}_{ij}^{(r)}$ is obtained from \eqref{RTT5n5} by replacing $T_{ij}^{(r)}$ by $\widetilde{T}_{ij}^{(r)}$,
\begin{equation}\label{RTT4n5}
\begin{split}
&q_i^{-\delta_{ik}}\widetilde{T}_{ij}^{(r)}\widetilde{T}_{kl}^{(s+1)}-q_i^{\delta_{ik}}\widetilde{T}_{ij}^{(r+1)}\widetilde{T}_{kl}^{(s)}
-\varepsilon_{ij;kl}\left(q_j^{-\delta_{jl}}\widetilde{T}_{kl}^{(s+1)}\widetilde{T}_{ij}^{(r)}
-q_j^{\delta_{jl}}\widetilde{T}_{kl}^{(s)}\widetilde{T}_{ij}^{(r+1)}\right) \\
=&\,\varepsilon_{ik;kl}\left(q_k-q_k^{-1}\right)\left(\delta_{k<i}\widetilde{T}_{kj}^{(r+1)}\widetilde{T}_{il}^{(s)}
-\delta_{j<l}\widetilde{T}_{kj}^{(s)}\widetilde{T}_{il}^{(r+1)}+\delta_{i<k}\widetilde{T}_{kj}^{(r)}\widetilde{T}_{il}^{(s+1)}
-\delta_{l<j}\widetilde{T}_{kj}^{(s+1)}\widetilde{T}_{il}^{(r)}\right).
\end{split}
\end{equation}

\begin{remark}
For $N=0$, we set $O$ to be the matrix $\sum\limits_{i=1}^ME_{i,M+1-i}$, then the $R$-matrix $\Re_q(u,v)$ satisfies the following equation
\begin{gather*}
(O\otimes O)\Re_q(u,v)(O\otimes O)=uvR\left(-u^{-1},-v^{-1}\right),
\end{gather*}
where $R(u,v)$ is the non-super trigonometric $R$-matrix (c.f. \cite[Formula 3.4]{MRS}). Then the algebra $X$ generated by the coefficients of the matrix elements of the matrices $OT\left(-u^{-1}\right)O$ and $O\widetilde{T}\left(-u^{-1}\right)O$ with relations \eqref{RTT1}$-$\eqref{RTT5} is isomorphic to $\mathrm{U}_q\big(\mathfrak{L}\mathfrak{gl}_M\big)$. Similarly, it coincides with the quantum loop algebra $\mathrm{U}_q\big(\mathfrak{L}\mathfrak{gl}_N\big)$ for $M=0$.
\end{remark}

\vspace{1em}
\subsubsection{Poincar\'e-Birkhoff-Witt theorem}
The Poincar\'e-Birkhoff-Witt theorem is an indispensable tool in studying the structures and representation theory of quantum (super-)groups. Numerous papers, including those referenced from \cite{BCFK, HZ}, heavily depend on this theorem for their analyses. However, a complete PBW basis for quantum affine (resp. loop) superalgebra of type $\mathbf{A}$ has not yet to be established.

In this subsection, we shall give a completed proof of the PBW theorem for the quantum loop superalgebra $\mathrm{U}_q\big(\mathfrak{L}\mathfrak{gl}_{M|N}\big)$ in RTT type presentation by using Gow-Molev's approach \cite[Sect. 2.1 \& 2.3]{GM} for quantum affine algebra $\mathrm{U}_q\big(\widehat{\mathfrak{gl}}_N\big)$ (see also \cite{BK} for Yangian $\mathrm{Y}\big(\mathfrak{gl}_N\big)$ and \cite{G} for super Yangian $\mathrm{Y}\big(\mathfrak{gl}_{M|N}\big)$ ). We start by reviewing some properties for the quantized enveloping superalgebra associated with Lie superalgebra $\mathfrak{gl}_{M|N}$. In this subsection, we set $q\in \mathbb{C}\backslash\{0\}$.

Set $\Re_q=\Re_q(1,0)$ and $\widetilde{\Re}_q=P\Re_q^{-1}P$, we obtain the relationships $\Re_q(u,v)=u\Re_q-v\widetilde{\Re}_q$ and $\Re_q-\widetilde{\Re}_q=\left(q-q^{-1}\right)P$. The matrix $\Re_q$ satisfies the quantum Yang-Baxter equation
\begin{gather*}
\Re_q^{12}\Re_q^{13}\Re_q^{23}=\Re_q^{23}\Re_q^{13}\Re_q^{12}.
\end{gather*}
The definition of quantized enveloping superalgebra $\mathrm{U}_q\big(\mathfrak{gl}_{M|N}\big)$ in terms of the $R$-matrix $\Re_q$ is given as follows.

\begin{definition}
The \textit{quantized enveloping superalgebra} $\mathrm{U}_q\big(\mathfrak{gl}_{M|N}\big)$ in its RTT type presentation is an associative superalgebra generated by the set of generators $\{T_{ij},\widetilde{T}_{ji}|1\leqslant i\leqslant j\leqslant M+N\}$ with the parity $|T_{ij}|=|\widetilde{T}_{ji}|=|i|+|j|$, satisfying the RTT-relations
\begin{align}\label{RTTu1}
&\hspace{3cm}T_{ii}\widetilde{T}_{ii}=\widetilde{T}_{ii}T_{ii}=1 \quad \text{for }i\in I, \\ \label{RTTu2}
&\Re_q^{12}T^{1}T^{2}=T^{2}T^{1}\Re_q^{12},\quad \Re_q^{12}\widetilde{T}^{1}\widetilde{T}^{2}=\widetilde{T}^{2}\widetilde{T}^{1}\Re_q^{12},\quad
\Re_q^{12}\widetilde{T}^{1}T^{2}=T^{2}\widetilde{T}^{1}\Re_q^{12},
\end{align}
where $\Re_q$ is the R-matrix given in subsection 2.2.1, and the upper(resp. lower) triangular matrix $T$\big(resp. $\widetilde{T}$\big) is defined by
\begin{gather*}
T=\sum_{i\leqslant j}E_{ij}\otimes T_{ij}\ \text{\big(resp. $\widetilde{T}=\sum_{i\leqslant j}E_{ij}\otimes T_{ij}$\big)}.
\end{gather*}

\end{definition}
It is a Hopf superalgebra endowed with the comultiplication $\Delta_0$ and antipode $S_0$ such that
\begin{gather*}
\Delta_0\left(T\right)=\sum_{k\in I}T\otimes T,\ \Delta_0\left(\widetilde{T}\right)=\sum_{k\in I}\widetilde{T}\otimes \widetilde{T};
\quad S_0\left(T\right)=T^{-1},\ S_0\left(\widetilde{T}\right)=\widetilde{T}^{-1}.
\end{gather*}
As known, the quantized enveloping superalgebra can be presented differently through Drinfeld-Jimbo type presentation, which has been subject to more extensive study. In this presentation (refer to \cite{HZ,Zr93} and the references therein), quantized enveloping superalgebra, denoted by $\mathcal{U}_q\big(\mathfrak{gl}_{M|N}\big)$, is generated by the set of generators
\begin{gather*}
\big\{e_{ij},k_a^{\pm 1}:=e_{aa}^{\pm 1}\big|i,j,a\in I,i\neq j\big\}\quad \text{with the parities }|e_{ij}|=|i|+|j|,\ |k_a|=0,
\end{gather*}
satisfying the relations
\begin{align}\label{DJu1}
&k_ak_a^{-1}=k_a^{-1}k_a=1,\quad k_ak_b=k_bk_a, \\ \label{DJu2}
&k_ae_{i,i+1}k_a^{-1}=q_a^{\delta_{ia}-\delta_{i+1,a}}e_{i,i+1},\quad k_ae_{i+1,i}k_a^{-1}=q_a^{\delta_{i+1,a}-\delta_{i,a}}e_{i+1,i}, \\ \label{DJu3}
&\left[e_{i,i+1},\,e_{j+1,j}\right]=\delta_{ij}\frac{k_ik_{i+1}^{-1}-k_i^{-1}k_{i+1}}{q_i-q_i^{-1}}, \\ \label{DJu4}
&\left(e_{M,M+1}\right)^2=\left(e_{M+1,M}\right)^2=0, \\ \label{DJu5}
&e_{i,i+1}e_{j,j+1}=e_{j,j+1}e_{i,i+1},\quad e_{i+1,i}e_{j+1,j}=e_{j+1,j}e_{i+1,i},\quad \text{if }j\neq i,i\pm 1, \\ \label{DJu6}
&\left(e_{i,i+1}\right)^2e_{i\pm 1,i\pm 1+1}-\left(q+q^{-1}\right)e_{i,i+1}e_{i\pm 1,i\pm 1+1}e_{i,i+1}+e_{i\pm 1,i\pm 1+1}\left(e_{i,i+1}\right)^2=0,\quad \text{if }i\neq M, \\ \label{DJu7}
&\left(e_{i+1,i}\right)^2e_{i\pm 1+1,i\pm 1}-\left(q+q^{-1}\right)e_{i+1,i}e_{i\pm 1+1,i\pm 1}e_{i+1,i}+e_{i\pm 1+1,i\pm 1}\left(e_{i+1,i}\right)^2=0,\quad \text{if }i\neq M, \\ \label{DJu8}
&\left[e_{M-1,M+2},\,e_{M,M+1}\right]=\left[e_{M+2,M-1},\,e_{M+1,M}\right]=0.
\end{align}
The elements $e_{ij}$ for $i<j$ can be defined recursively by
\begin{gather}\label{DJu9}
e_{ij}=e_{ik}e_{kj}-q_k^{-1}e_{kj}e_{ik},\quad e_{ji}=e_{jk}e_{ki}-q_ke_{ki}e_{jk}\quad\text{for }i<k<j.
\end{gather}
The following proposition has been adopted in numerous papers, as exhibited by \cite[Proposition 3.3]{Zh}.
\begin{proposition}\label{facts}
  \begin{enumerate}
  \item[(1)] There exists an isomorphism of Hopf superalgebras $\Psi:\mathrm{U}_q\big(\mathfrak{gl}_{M|N}\big)\rightarrow \mathcal{U}_q\big(\mathfrak{gl}_{M|N}\big)$ such that
      \begin{gather}\label{isou1}
      T_{ii}=\widetilde{T}_{ii}^{-1}\mapsto k_i,\ T_{ij}\mapsto (1-q_i^{-2})k_ie_{ij},\quad \widetilde{T}_{ji}\mapsto (1-q_i^2)e_{ji}k_i^{-1}.
      \end{gather}
  \item[(2)] The map $T_{ij}\mapsto T_{ij}^{(0)},\ \widetilde{T}_{ji}\mapsto \widetilde{T}_{ji}^{(0)}$
  can be uniquely extended to a homomorphism of Hopf superalgebras ${\bf em}:\mathrm{U}_q\big(\mathfrak{gl}_{M|N}\big)\rightarrow \mathrm{U}_q\big(\mathfrak{L}\mathfrak{gl}_{M|N}\big)$.
  \item[(3)] The map
   $T_{ij}(u)\mapsto T_{ij}-\widetilde{T}_{ij}u,\ \widetilde{T}_{ij}(u)\mapsto \widetilde{T}_{ij}-T_{ij}u^{-1}$
   can be uniquely extended to a homomorphism of Hopf superalgebras ${\bf ev}:\mathrm{U}_q\big(\mathfrak{L}\mathfrak{gl}_{M|N}\big)\rightarrow \mathrm{U}_q\big(\mathfrak{gl}_{M|N}\big)$.
  \end{enumerate}
\end{proposition}
The map defined in the last statement of Proposition \ref{facts} is often called by \textit{evaluation homomorphism}. Observing that the composite ${\bf em}\circ {\bf ev}$ is the identity map on $\mathrm{U}_q\big(\mathfrak{gl}_{M|N}\big)$. The PBW theorem of $\mathrm{U}_q\big(\mathfrak{gl}_{M|N}\big)$ depends on the first statement in Proposition \ref{facts}. For a more detailed proof, it is necessary to introduce some new superalgebras.

Consider an indeterminate $z$ and set $z_i=z^{(-1)^{|i|}}$. We introduce the associative superalgebra $\mathrm{U}_z\big(\mathfrak{gl}_{M|N}\big)$ over $\mathbb{C}(z)$ generated by the set $\big\{T_{ij},\widetilde{T}_{ji}\big|1\leqslant i\leqslant j\leqslant M+N\big\}$ with the parities $|T_{ij}|=|\widetilde{T}_{ji}|=|i|+|j|$. The defining relations of this superalgebra are obtained from \eqref{RTTu1} and \eqref{RTTu2} by replacing $q$ with $z$. Similarly, we introduce another associative $\mathbb{C}(z)$-superalgebra $\mathcal{U}_z\big(\mathfrak{gl}_{M|N}\big)$ generated by the set $\big\{e_{ij},k_a^{\pm 1}\big|i,j,a\in I,i\neq j\big\}$ with the parities $|e_{ij}|=|i|+|j|,\ |k_i|=0$ concerning the relations obtained from \eqref{DJu1}--\eqref{DJu9} by replacing $q$ with $z$.
The isomorphism between these two superalgebras is also provided as \eqref{isou1}. Here we adopt the same notations for the generators without confusion.

Now, define the superalgebra $\mathrm{U}_z^{\circ}\big(\mathfrak{gl}_{M|N}\big)$ over $\mathbb{C}[z,z^{-1}]$ as $\mathrm{U}_z\big(\mathfrak{gl}_{M|N}\big)$ by replacing $\mathbb{C}(z)$ with $\mathbb{C}[z,z^{-1}]$. Then we have the following isomorphism
\begin{gather}\label{isou2}
\mathrm{U}_z^{\circ}\big(\mathfrak{gl}_{M|N}\big)\otimes_{\mathbb{C}[z,z^{-1}]}\mathbb{C}\xlongrightarrow{\simeq} \mathrm{U}_q\big(\mathfrak{gl}_{M|N}\big),
\end{gather}
where $\mathbb{C}$ is viewed as a $\mathbb{C}[z,z^{-1}]$-module through the evaluation $z=q$.

Define an order on the set of generators for $\mathrm{U}_z^{\circ}\big(\mathfrak{gl}_{M|N}\big)$ (also for $\mathrm{U}_q\big(\mathfrak{gl}_{M|N}\big)$):
\begin{align*}
&T_{i_1,j_1}\preceq T_{i_2,j_2}\quad\text{if and only if}\quad (i_1,j_1;0)\preceq (i_2,j_2;0), \\
&\widetilde{T}_{i_1,j_1}\prec \widetilde{T}_{i_2,j_2}\quad\text{if and only if}\quad (i_1,j_1;0)\preceq (i_2,j_2;0), \\
&T_{i_1,j_1}\preceq \widetilde{T}_{i_2,j_2}\quad\text{if and only if}\quad (i_1,j_1;0)\preceq (i_2,j_2;0).
\end{align*}

\begin{proposition}
Let $\mathfrak{B}$ be the set of all ordered monomials
\begin{align*}
&\left(T_{1,1}\right)^{m_{1,1}}\left(\widetilde{T}_{1,1}\right)^{\widetilde{m}_{1,1}}
\left(T_{1,2}\right)^{m_{1,2}}\cdots \left(T_{1,M+N}\right)^{m_{1,M+N}}\cdots \\
\cdots&\left(\widetilde{T}_{M+N,1}\right)^{\widetilde{m}_{M+N,1}}\cdots \left(T_{M+N,M+N}\right)^{m_{M+N,M+N}}\left(\widetilde{T}_{M+N,M+N}\right)^{\widetilde{m}_{M+N,M+N}}
\end{align*}
with the exponents
\begin{align*}
&m_{ij},\widetilde{m}_{ij}\in\mathbb{Z}_{\geqslant 0}\quad \text{for }|i_a|+|j_a|=\bar{0},\ i_a\neq j_a, \\
&m_{ij},\widetilde{m}_{ij}\in\{0,1\}\quad \text{for }|i_a|+|j_a|=\bar{1}, \\
&m_{ii}\widetilde{m}_{ii}=0\quad \text{for } i\in I.
\end{align*}
Then $\mathfrak{B}$ forms a basis of $\ \mathrm{U}_z^{\circ}\big(\mathfrak{gl}_{M|N}\big)$ over $\mathbb{C}[z,z^{-1}]$.
\end{proposition}

\begin{proof}
We rewrite the last relation of \eqref{RTTu2} for $\mathrm{U}_z^{\circ}\big(\mathfrak{gl}_{M|N}\big)$ as more concretely relations in terms of generators $T_{ij},\widetilde{T}_{ij}$ given by
\begin{gather}\label{zueq1}
z_i^{\delta_{ik}}\widetilde{T}_{ij}T_{kl}-\varepsilon_{ij;kl}z_j^{\delta_{jl}}T_{kl}\widetilde{T}_{ij}=\varepsilon_{ik;kl}\left(z_k-z_k^{-1}\right)\times
\left(\delta_{j<l}T_{kj}\widetilde{T}_{il}-\delta_{k<i}\widetilde{T}_{kj}T_{il}\right).
\end{gather}
If $i>k$, we have $\widetilde{T}_{ij}T_{kl}\in\operatorname{span}_{\mathbb{C}[z,z^{-1}]}\mathfrak{B}$. Take $i=k$ and for $l>j$, then
\begin{gather*}
T_{il}\widetilde{T}_{ij}=-\varepsilon_{ji;il}z_j^{-\delta_{jl}}\bigg\{\widetilde{T}_{ij}T_{il}
+\left(z_i-z_i^{-1}\right)\delta_{j<l}T_{ij}\widetilde{T}_{il}\bigg\}\in\operatorname{span}_{\mathbb{C}[z,z^{-1}]}\mathfrak{B}.
\end{gather*}
The relation between the generators $T_{ij},T_{kl}$ (resp. $\widetilde{T}_{ji},\widetilde{T}_{lk}$) can be deduced from \eqref{zueq1} by replacing $\widetilde{T}$ (resp. $T$) by $T$ (resp. $\widetilde{T}$). This allows us to apply the same arguments for these relations.
Therefore, the set $\mathfrak{B}$ spans $\mathrm{U}_z^{\circ}\big(\mathfrak{gl}_{M|N}\big)$.

We are left to prove that $\mathfrak{B}$ is linearly independent. Suppose not, there exists a linear combination $\mathcal{M}$ with nonzero coefficients in $\mathbb{C}[z,z^{-1}]$ equal to zero. Using the isomorphism \eqref{isou1}, $\mathcal{M}$ is a linear combination of $\Psi\left(\mathfrak{B}\right)$ with nonzero coefficients in $\mathbb{C}[z,z^{-1}]$ equal to zero, which is contradicted to the PBW theorem of $\mathcal{U}_z\big(\mathfrak{gl}_{M|N}\big)$(see more details in \cite[Theorem 4.1]{HZ}). Here we use our order on the generators of $\mathcal{U}_z\big(\mathfrak{gl}_{M|N}\big)$(via the isomorphism \eqref{isou1}, we identify with the corresponding generators), which also constructs a PBW type basis of $\mathcal{U}_z\big(\mathfrak{gl}_{M|N}\big)$.

\end{proof}
Then we immediately have
\begin{corollary}
Let $\rho$ be the isomorphism given in \eqref{isou2}. Then $\rho\big(\mathfrak{B}\otimes 1\big)$ forms a basis of $\ \mathrm{U}_q\big(\mathfrak{gl}_{M|N}\big)$.
\end{corollary}

Note that the superalgebra $\mathrm{U}_1\big(\mathfrak{gl}_{M|N}\big)$ is supercommutative when $q=1$. Then we can identify it with the supersymmetric algebra $\mathfrak{S}$ generated by the set $\big\{g_{ij},\widetilde{g}_{ji}\big|1\leqslant i\leqslant j\leqslant M+N\big\}$ respect to the relations $g_{ii}\widetilde{g}_{ii}=1$ for all $i\in I$ and the parities $|g_{ij}|=|\widetilde{g}_{ji}|=|i|+|j|$. Then we have
\begin{gather}\label{isou3}
\mathrm{U}_z^{\circ}\big(\mathfrak{gl}_{M|N}\big)\otimes_{\mathbb{C}[z,z^{-1}]}\mathbb{C}\xlongrightarrow{\simeq} \mathfrak{S},
\end{gather}
where $\mathbb{C}$ is viewed as a $\mathbb{C}[z,z^{-1}]$-module via taking $z=1$.

\vspace{1em}
Let us return to the quantum loop superalgebra. We also define the superalgebra $\mathrm{U}_z^{\circ}\big(\mathfrak{L}\mathfrak{gl}_{M|N}\big)$ over $\mathbb{C}[z,z^{-1}]$ generated by the set $\big\{T_{ij}^{(r)},\widetilde{T}_{ij}^{(r)}\big|i,j\in I,\ r\in\mathbb{Z}_{\geqslant 0}\big\}$ with the parities $|T_{ij}^{(r)}|=|\widetilde{T}_{ij}^{(r)}|=|i|+|j|$ and with the relations obtained from \eqref{RTT1}--\eqref{RTT5} by replacing $q\rightarrow z$. Here $z$ is an indeterminate as before. Then we also have
\begin{gather}\label{isou4}
\mathrm{U}_z^{\circ}\big(\mathfrak{L}\mathfrak{gl}_{M|N}\big)\otimes_{\mathbb{C}[z,z^{-1}]}\mathbb{C}\xlongrightarrow{\simeq} \mathrm{U}_q(\mathfrak{Lgl}_{M|N}),
\end{gather}
where $\mathbb{C}$ is viewed as a $\mathbb{C}[z,z^{-1}]$-module via taking $z=q$.

It is clear that $\mathrm{U}_z^{\circ}\big(\mathfrak{L}\mathfrak{gl}_{M|N}\big)$ also has a Hopf superalgebra structure over $\mathbb{C}[z,z^{-1}]$ with the same comultiplication and antipode as $\mathrm{U}_q\big(\mathfrak{L}\mathfrak{gl}_{M|N}\big)$, denoted by $\Delta_z$ and $S_z$ respectively. Define $\Delta_z^{(p)}$($p>0$) by the $p-1$ times composite
\begin{gather*}
\left(\Delta_z\otimes 1^{\otimes p-2}\right)\circ\cdots\circ\left(\Delta_z\otimes 1\right)\circ\Delta_z,
\end{gather*}
which maps $\mathrm{U}_z^{\circ}\big(\mathfrak{L}\mathfrak{gl}_{M|N}\big)$ onto $\mathrm{U}_z^{\circ}\big(\mathfrak{L}\mathfrak{gl}_{M|N}\big)^{\otimes p}$
such that
\begin{align*}
&T_{ij}^{(r)}\mapsto \sum_{a_1+\cdots+a_p=r}\sum_{k_1,\ldots,k_{p-1}\in I}(-1)^dT_{ik_1}^{(a_1)}\otimes T_{k_1k_2}^{(a_2)}\otimes\cdots\otimes T_{k_{p-1}j}^{(a_p)}, \\
&\widetilde{T}_{ij}^{(r)}\mapsto \sum_{a_1+\cdots+a_p=r}\sum_{k_1,\ldots,k_{p-1}\in I}(-1)^d\widetilde{T}_{ik_1}^{(a_1)}\otimes\widetilde{T}_{k_1k_2}^{(a_2)}\otimes\cdots\otimes \widetilde{T}_{k_{p-1}j}^{(a_p)},
\end{align*}
where $d=|i||j|+|i||k_1|+|k_1|+|j||k_{p-1}|+\sum_{t=2}^{p-1}|k_s|\big(|k_s|+|k_{s-1}|\big)$ for $p>1$ and $d=0$ for $p=1$. Moreover, the map $\textbf{ev}_z:\mathrm{U}_z^{\circ}\big(\mathfrak{L}\mathfrak{gl}_{M|N}\big)\rightarrow \mathrm{U}_z^{\circ}\big(\mathfrak{gl}_{M|N}\big)$ can be reduced from the evaluation $\textbf{ev}$ defined in Proposition \ref{facts}.

Introduce a homomorphism of superalgebras over $\mathbb{C}[z,z^{-1}]$ for every $p\in\mathbb{Z}_{>0}$
\begin{gather*}
\kappa_p:\mathrm{U}_z^{\circ}\big(\mathfrak{L}\mathfrak{gl}_{M|N}\big)
\rightarrow\mathrm{U}_z^{\circ}\big(\mathfrak{gl}_{M|N}\big)^{\otimes p}
\end{gather*}
given by
\begin{gather*}
\kappa_p=\big(\textbf{ev}_z\otimes \textbf{ev}_z\otimes \cdots \otimes \textbf{ev}_z\big)\circ \Delta_z^{(p)}.
\end{gather*}
More precisely, the map $\kappa_p$ actions on the generators of $\mathrm{U}_z^{\circ}\big(\mathfrak{L}\mathfrak{gl}_{M|N}\big)$ via
\begin{align*}
&T_{ij}^{(r)}\mapsto \sum_{k_1,\ldots,k_{p-1}}\sum_{1\leqslant s_1<\cdots<s_r\leqslant p}(-1)^{d+r}T_{ik_1}^{[1]}T_{k_1k_2}^{[2]}\cdots \widetilde{T}_{k_{s_1-1}k_{s_1}}^{[s_1]}\cdots \widetilde{T}_{k_{s_r-1}k_{s_r}}^{[s_r]}\cdots T_{k_{p-1}j}^{[p]}, \\
&\widetilde{T}_{ij}^{(r)}\mapsto \sum_{k_1,\ldots,k_{p-1}}\sum_{1\leqslant s_1<\cdots<s_r\leqslant p}(-1)^{d+r}\widetilde{T}_{ik_1}^{[1]}\widetilde{T}_{k_1k_2}^{[2]}\cdots T_{k_{s_1-1}k_{s_1}}^{[s_1]}\cdots T_{k_{s_r-1}k_{s_r}}^{[s_r]}\cdots \widetilde{T}_{k_{p-1}j}^{[p]}, \\
\end{align*}
where the index subset $\left\{k_0=i,k_1,\ldots,k_{p-1},k_p=j\right\}\subset I$ in the first (resp. second) relation satisfies the condition
\begin{gather*}
k_{s_{a-1}}\geqslant k_{s_a},\ k_{b-1}\leqslant k_b\ \text{(resp. $k_{s_{a-1}}\leqslant k_{s_a},\ k_{b-1}\geqslant k_b$\ )}
\end{gather*}
for $1\leqslant a\leqslant r$ and $b\in \{1,\ldots,p\}\setminus\{s_1,\ldots,s_r\}$ and the notations $X_{ij}^{[s]}$ for $X\in\left\{T,\widetilde{T}\right\}$ means
$1^{\otimes (s-1)}\otimes X_{ij}\otimes 1^{\otimes (p-s)}$. Notice that the images of $T_{ij}^{(r)}$ and $\widetilde{T}_{ij}^{(r)}$ under $\kappa_p$ are zero for the case of $p<r$.

Define an order on the set generators both for $\mathrm{U}_z^{\circ}\big(\mathfrak{L}\mathfrak{gl}_{M|N}\big)$ and for $\mathrm{U}_q\big(\mathfrak{L}\mathfrak{gl}_{M|N}\big)$:
\begin{align*}
&T_{i_1,j_1}^{(r_1)}\preceq T_{i_2,j_2}^{(r_2)}\quad\text{if and only if}\quad (i_1,j_1;r_1)\preceq (i_2,j_2;r_2), \\
&\widetilde{T}_{i_1,j_1}^{(r_1)}\preceq \widetilde{T}_{i_2,j_2}^{(r_2)}\quad\text{if and only if}\quad (i_1,j_1;r_1)\preceq (i_2,j_2;r_2), \\
&T_{i_1,j_1}^{(r_1)}\preceq \widetilde{T}_{i_2,j_2}^{(r_2)}\quad\text{if and only if}\quad (i_1,j_1;r_1)\preceq (i_2,j_2;r_2).
\end{align*}

\begin{proposition}\label{base}
Let $\widehat{\mathfrak{B}}$ be the set of all ordered monomials (set $K=M+N$)
\begin{equation}\label{base:Uq}
\begin{split}
&\left(T_{1,1}^{(0)}\right)^{m_{1,1;0}}\left(\widetilde{T}_{1,1}^{(0)}\right)^{\widetilde{m}_{1,1;0}}
\left(T_{1,1}^{(1)}\right)^{m_{1,1;1}}\left(\widetilde{T}_{1,1}^{(1)}\right)^{\widetilde{m}_{1,1;1}}\cdots \\
\cdots&\left(T_{1,2}^{(0)}\right)^{m_{1,2;0}}
\left(T_{1,2}^{(1)}\right)^{m_{1,2;1}}\left(\widetilde{T}_{1,2}^{(1)}\right)^{\widetilde{m}_{1,2;1}}\cdots
\left(T_{1,K}^{(0)}\right)^{m_{1,K;0}}
\left(T_{1,K}^{(1)}\right)^{m_{1,K;1}}\left(\widetilde{T}_{1,K}^{(1)}\right)^{\widetilde{m}_{1,K;1}}\cdots  \\
\cdots&\left(\widetilde{T}_{2,1}^{(0)}\right)^{\widetilde{m}_{2,1;0}}
\left(T_{2,1}^{(1)}\right)^{m_{2,1;1}}\left(\widetilde{T}_{2,1}^{(1)}\right)^{\widetilde{m}_{2,1;1}}\cdots
\left(\widetilde{T}_{2,K}^{(0)}\right)^{\widetilde{m}_{2,K;0}}
\left(T_{2,K}^{(1)}\right)^{m_{2,K;1}}\left(\widetilde{T}_{2,K}^{(1)}\right)^{\widetilde{m}_{2,K;1}}\cdots \\
\cdots&\left(T_{K,K}^{(0)}\right)^{m_{K,K;0}}\left(\widetilde{T}_{K,K}^{(0)}\right)^{\widetilde{m}_{K,K;0}}
\left(T_{K,K}^{(1)}\right)^{m_{K,K;1}}\left(\widetilde{T}_{K,K}^{(1)}\right)^{\widetilde{m}_{K,K;1}}\cdots
\end{split}
\end{equation}
with the exponents
\begin{align}\label{base:1}
&m_{i,j;r},\ \widetilde{m}_{i,j;r}\in\mathbb{Z}_{\geqslant 0}\quad \text{for }|i|+|j|=\bar{0}, \\ \label{base:2}
&m_{i,j;r},\ \widetilde{m}_{i,j;r}\in\{0,1\}\quad \text{for }|i|+|j|=\bar{1}, \\ \label{base:3}
&m_{i,i;0}\widetilde{m}_{i,i;0}=0\quad\text{for }i\in I.
\end{align}
Then $\widehat{\mathfrak{B}}$ forms a basis of $\ \mathrm{U}_z^{\circ}\big(\mathfrak{L}\mathfrak{gl}_{M|N}\big)$ over $\mathbb{C}[z,z^{-1}]$.
\end{proposition}

\begin{proof}
First, we claim that the monomial set $\widehat{\mathfrak{B}}$ spans $\mathrm{U}_z^{\circ}\big(\mathfrak{L}\mathfrak{gl}_{M|N}\big)$. It is enough to show that any monomials with the form $\eta_{i_a,j_a}^{(r_a)}\eta_{i_b,j_b}^{(r_b)}$ for $(i_b,j_b;r_b)\prec (i_a,j_a;r_a)$ can be regarded as a nontrivial combination of $\widehat{\mathfrak{B}}$. The relation $\eqref{RTT5}$ is equivalent to
\begin{equation}\label{zueq2}
\begin{split}
&\left(z_i^{-\delta{ik}}v-z_i^{\delta{ik}}u\right)\widetilde{T}_{ij}(u)T_{kl}(v)
-\varepsilon_{ij;kl}\left(z_j^{-\delta{jl}}v-z_j^{\delta{jl}}u\right)T_{kl}(u)\widetilde{T}_{ij}(v) \\
=&\varepsilon_{ik;kl}\left(z_k-z_k^{-1}\right)\left\{\left(\delta_{k<i}u+\delta_{i<k}v\right)\widetilde{T}_{kj}(u)T_{il}(v)
-\left(\delta_{j<l}u+\delta_{l<j}v\right)T_{kj}(v)\widetilde{T}_{il}(u)\right\}.
\end{split}
\end{equation}
Here we replace $q$ with the indeterminate $z$. It is clear that the coefficients of $\widetilde{T}_{ij}(u)T_{kl}(v)$ for $i>k$ are contained in $\operatorname{span}_{\mathbb{C}[z,z^{-1}]}\widehat{\mathfrak{B}}$. Take $i=k$ and $l>j$, then we have
\begin{gather*}
T_{il}(u)\widetilde{T}_{ij}(v)
=\varepsilon_{ji;il}\frac{z_i^{-1}v-z_iu}{v-u}\widetilde{T}_{ij}(u)T_{il}(v)
+\varepsilon_{ji;il}\frac{\left(z_i-z_i^{-1}\right)u}{v-u}T_{ij}(v)\widetilde{T}_{il}(u),
\end{gather*}
which implies that the coefficients of $\widetilde{T}_{il}(u)T_{ij}(v)$ for $l>j$ are also contained in $\operatorname{span}_{\mathbb{C}[z,z^{-1}]}\widehat{\mathfrak{B}}$. Furthermore, we have
\begin{gather*}
\widetilde{T}_{ij}^{(r)}T_{ij}^{(s)}=(-1)^{|i|+|j|}\frac{z_j^{-1}v-z_ju}{z_i^{-1}v-z_iu}T_{ij}^{(s)}\widetilde{T}_{ij}^{(r)}.
\end{gather*}
By applying the Taylor series expansion
\begin{gather*}
\frac{1}{z_i^{-1}v-z_iu}=\sum_{t=0}^{\infty}z_i^{2t+1}u^tv^{-(t+1)},
\end{gather*}
one gets for $r\geqslant s$
\begin{gather*}
\widetilde{T}_{ij}^{(r)}T_{ij}^{(s)}=(-1)^{|i|+|j|}\sum_{t=0}^{\infty}z_i^{2t+1}
\left(z_j^{-1}T_{ij}^{(s-t)}\widetilde{T}_{ij}^{(r-t)}-z_jT_{ij}^{(s-t-1)}\widetilde{T}_{ij}^{(r-t-1)}\right)\in \operatorname{span}_{\mathbb{C}[z,z^{-1}]}\widehat{\mathfrak{B}}.
\end{gather*}
The relations between the generators $T_{ij}^{(r)},T_{kl}^{(s)}$ (resp. $\widetilde{T}_{ij}^{(r)},\widetilde{T}_{kl}^{(s)}$) can be deduced from \eqref{zueq2} by replacing $\widetilde{T}$ (resp. $T$) by $T$ (resp. $\widetilde{T}$). This allows us to apply the same arguments for these relations.

Next, we need to show that $\widehat{\mathfrak{B}}$ is linearly independent. Suppose not, there exists a linear combination $\mathcal{N}$ with nonzero coefficients in $\mathbb{C}[z,z^{-1}]$ equal to zero. The number of the pairwise distinct generators $T_{ij}^{(r)},\widetilde{T}_{ij}^{(r)}$ occurred in $\mathcal{N}$ must be finite. Choose a suitable positive integer $t$ not less than any possible $r$ for $T_{ij}^{(r)},\widetilde{T}_{ij}^{(r)}$ in $\mathcal{N}$. Put $p=2t+1$. Every term in $\mathcal{N}$ is not vanish under the homomorphism $\kappa_p$. Then we get a nontrivial $\mathbb{C}[z,z^{-1}]$-linear combination $\kappa_p\left(\mathcal{N}\right)$ in the superalgebra $\mathrm{U}_z^{\circ}\big(\mathfrak{gl}_{M|N}\big)^{\otimes p}$ equal to zero.

Recall the supersymmetric algebra $\mathfrak{S}$ with generators $g_{ij},\widetilde{g}_{ij}$ introduced as before. We regard $\mathfrak{S}^{\otimes p}$ as the algebra with the supersymmetric generators $g_{ij}^{[s]},\widetilde{g}_{ij}^{[s]}$, where we use the same notations $X_{ij}^{[s]}$ for $X=g$ or $\widetilde{g}$. Denote the elements $f_{ij}^{(r)}$\bigg(resp. $\widetilde{f}_{ij}^{(r)}$\bigg) by the image of $\kappa_p\left(T_{ij}^{(r)}\right)\otimes 1$\bigg(resp. $\kappa_p\left(\widetilde{T}_{ij}^{(r)}\right)\otimes 1$\bigg) under the isomorphism \eqref{isou3}. Then we have a nontrivial $\mathbb{C}$-linear combination in $f_{ij}^{(r)},\widetilde{f}_{ij}^{(r)}$ equal to zero.

Denote this $\mathbb{C}$-linear combination by $\overline{\mathcal{N}}$. We claim that the coefficients of $\overline{\mathcal{N}}$ are all zero. To show it, we need to prove that the same ordered monomials as $\widehat{\mathfrak{B}}$ in $f_{ij}^{(s)},\widetilde{f}_{ij}^{(s)}$ are linearly independent. Here we identify $\widetilde{f}_{ii}^{(0)}$ with the inverse of $f_{ii}^{(0)}$ for all $i\in I$. It is sufficient to show that the morphism $\Theta$ mapping $\left(g_{ij}^{[s]},\,\widetilde{g}_{ij}^{[s]}\right)$ to $\left(f_{ij}^{(r)},\,\widetilde{f}_{ij}^{(r)}\right)$ for $1\leqslant p$ and $0\leqslant r\leqslant t$ is injective.

Let $J_{x_1\cdots x_k}$ be the following $k\times k$ matrix
\begin{equation*}
\begin{pmatrix}
1 & 1 &\cdots & 1 \\
\sum_{i\neq 1} x_i & \sum_{i\neq 2} x_i &\cdots & \sum_{i\neq k} x_i \\
\cdots & \cdots &\cdots \\
\sum_{i_1,\ldots,i_{k-1}\neq 1} x_{i_1}x_{i_2}\cdots x_{i_{k-1}} & \sum_{i_1,\ldots,i_{k-1}\neq 2} x_{i_1}x_{i_2}\cdots x_{i_{k-1}}
 &\cdots & \sum_{i_1,\ldots,i_{k-1}\neq k} x_{i_1}x_{i_2}\cdots x_{i_{k-1}}
\end{pmatrix}
\end{equation*}
over $\mathbb{C}$. Consider the differential $\operatorname{d}\Theta$ which maps $\left(\operatorname{d}g_{ij}^{[s]},\,\operatorname{d}\widetilde{g}_{ij}^{[s]}\right)$ to $\left(\operatorname{d}f_{ij}^{(r)},\,\operatorname{d}\widetilde{f}_{ij}^{(r)}\right)$. We are left to show that the matrix of $\operatorname{d}\Theta$ is nondegenerate at some point in $\mathfrak{S}^{\otimes p}$. Pick nonzero pairwise distinct complex number $c_1,\cdots,c_p$ and set
\begin{gather*}
\left(g_{ij}^{[s]},\,\widetilde{g}_{ij}^{[s]}\right)=\left((-1)^{|i||j|+|j|}\delta_{ij}c_s,\,(-1)^{|i||j|+|j|}\delta_{ij}c_s^{-1}\right)
\end{gather*}
for all $1\leqslant s\leqslant p$. Then we have
\begin{equation*}
\begin{pmatrix}
\operatorname{d}f_{ij}^{(0)} \\
\vdots \\
\operatorname{d}f_{ij}^{(t)} \\
\operatorname{d}\widetilde{f}_{ij}^{(t)} \\
\vdots \\
\operatorname{d}\widetilde{f}_{ij}^{(1)}
\end{pmatrix}
=(-1)^{m(m+1)}\left(c_1\cdots c_p\right)^{p-1}
\cdot J_{c_1^{-2}\cdots c_p^{-2}}\cdot
\begin{pmatrix}
\operatorname{d}g_{ij}^{[1]} \\
\vdots \\
\operatorname{d}g_{ij}^{[p]}
\end{pmatrix}
\end{equation*}
for $i<j$,
\begin{equation*}
\begin{pmatrix}
\operatorname{d}\widetilde{f}_{ij}^{(0)} \\
\vdots \\
\operatorname{d}\widetilde{f}_{ij}^{(t)} \\
\operatorname{d}f_{ij}^{(t)} \\
\vdots \\
\operatorname{d}f_{ij}^{(1)}
\end{pmatrix}
=(-1)^{m(m+1)}\left(c_1\cdots c_p\right)^{-(p-1)}
\cdot J_{c_1^2\cdots c_p^2}
\cdot
\begin{pmatrix}
\operatorname{d}\widetilde{g}_{ij}^{[1]} \\
\vdots \\
\operatorname{d}\widetilde{g}_{ij}^{[p]}
\end{pmatrix}
\end{equation*}
for $i>j$, and
\begin{equation*}
\begin{pmatrix}
\operatorname{d}f_{ii}^{(0)} \\
\vdots \\
\operatorname{d}f_{ij}^{(t)} \\
\operatorname{d}\widetilde{f}_{ij}^{(t)} \\
\vdots \\
\operatorname{d}\widetilde{f}_{ii}^{(1)}
\end{pmatrix}
=(-1)^{m(m+1)}\left\{\left(c_1c_2\cdots c_p\right)^{p-1}-
\left(c_1c_2\cdots c_p\right)^{p-3}\right\}\cdot J_{c_1^{-2}\cdots c_p^{-2}}\cdot
\begin{pmatrix}
\operatorname{d}g_{ii}^{[1]} \\
\vdots \\
\operatorname{d}g_{ii}^{[p]}
\end{pmatrix}.
\end{equation*}
Every coefficient matrix of the three cases is nondegenerate since the determinant of the matrix $J_{c_1\cdots c_p}$ is nonzero. Therefore, the superderivations $\operatorname{d}f_{ii}^{(r)},\operatorname{d}\widetilde{f}_{ii}^{(r)}$ are linearly independent, which implies the injectivity of $\Theta$ and then proves the claim.

The statements above indicate that the coefficients in $\mathcal{N}$ are all zero, which completes the proof.

\end{proof}

The next corollary gives a PBW type basis of quantum loop superalgebra $\mathrm{U}_q\big(\mathfrak{L}\mathfrak{gl}_{M|N}\big)$.
\begin{corollary}\label{PBWaf}
Let $\hat{\rho}$ be the isomorphism given in \eqref{isou4}. Then $\hat{\rho}\big(\widehat{\mathfrak{B}}\otimes 1\big)$ forms a basis of $\ \mathrm{U}_q\big(\mathfrak{L}\mathfrak{gl}_{M|N}\big)$.
\end{corollary}

\vspace{1em}
\subsubsection{$\mathbb{A}$-form and the classical limit}
Let $\mathbb{A}$ be the localization of $\mathbb{C}[q]$ at the ideal $(q-1)$. To be more exact,
\begin{gather*}
\mathbb{A}=\Bigg\{\,\frac{f(q)}{g(q)}\,\Bigg|\,f(q),g(q)\in \mathbb{C}[q],\ g(1)\neq 0\,\Bigg\}.
\end{gather*}
Let $\mathrm{U}_{\mathbb{A}}$ be the $\mathbb{A}$-sub-superalgebra with the generators $\gamma_{ij}^{(r)},\widetilde{\gamma}_{ij}^{(r)}$ given by
\begin{equation*}
\gamma_{ij}^{(r)}=\begin{cases}
\displaystyle\frac{T_{ii}^{(0)}-1}{q-1}, &\textrm{if}~~i=j,\ r=0, \\
& \\
\displaystyle\frac{T_{ij}^{(r)}}{q-q^{-1}}, &\textrm{otherwise},
\end{cases}\quad
\widetilde{\gamma}_{ij}^{(r)}=\begin{cases}
\displaystyle\frac{\widetilde{T}_{ii}^{(0)}-1}{q-1}, &\textrm{if}~~i=j,\ r=0, \\
& \\
\displaystyle\frac{\widetilde{T}_{ij}^{(r)}}{q-q^{-1}}, &\textrm{otherwise}.
\end{cases}
\end{equation*}
Observe that all elements $T_{ij}^{(r)},\widetilde{T}_{ij}^{(r)}\in\mathrm{U}_{\mathbb{A}}$. The diagonal entries in $\mathrm{U}_{\mathbb{A}}$ satisfy
\begin{gather}\label{modulo:1}
\gamma_{ii}^{(0)}+\widetilde{\gamma}_{ii}^{(0)}=-(q-1)\gamma_{ii}^{(0)}\widetilde{\gamma}_{ii}^{(0)}\in (q-1)\mathrm{U}_{\mathbb{A}},
\end{gather}
and
\begin{gather}\label{modulo:2}
T_{ii}^{(0)}-1=(q-1)\gamma_{ii}^{(0)},\quad \widetilde{T}_{ii}^{(0)}-1=(q-1)\widetilde{\gamma}_{ii}^{(0)}\in (q-1)\mathrm{U}_{\mathbb{A}}.
\end{gather}

\begin{proposition}
The sub-superalgebra $\mathrm{U}_{\mathbb{A}}$ is an $\mathbb{A}$-form of $\ \mathrm{U}_q\big(\mathfrak{L}\mathfrak{gl}_{M|N}\big)$.
\end{proposition}
\begin{proof}

It suffices to prove that the map
\begin{gather}\label{iso:ua}
\mathrm{U}_{\mathbb{A}}\otimes_{\mathbb{A}} \mathbb{C}(q)\xlongrightarrow{\simeq} \mathrm{U}_q\big(\mathfrak{Lgl}_{M|N}\big)
\end{gather}
is an isomorphism over $\mathbb{C}(q)$, where $\mathbb{C}(q)$ is considered as an $\mathbb{A}$-module. By Corollary \ref{PBWaf}, the set $\widehat{\mathfrak{B}}$ forms a $\mathbb{C}$-basis of $\mathrm{U}_q\big(\mathfrak{Lgl}_{M|N}\big)$.

Let $\widehat{\mathfrak{B}}_1$ be the set obtained from $\widehat{\mathfrak{B}}$ by replacing $T_{ij}^{(r)},\widetilde{T}_{ij}^{(r)}$ with $\gamma_{ij}^{(r)},\widetilde{\gamma}_{ij}^{(r)}$. We claim that $\widehat{\mathfrak{B}}_1$ also forms a basis of $\mathrm{U}_q\big(\mathfrak{Lgl}_{M|N}\big)$ over $\mathbb{C}(q)$. Given a monomial
\begin{gather}\label{monomial:p}
x_{i_1,j_1}^{(r_1)}x_{i_2,j_2}^{(r_2)}\cdots x_{i_p,j_p}^{(r_p)},
\end{gather}
we define the positive integer $p$ as the length of this monomial. Introduce the filtration on $\mathrm{U}_q\big(\mathfrak{Lgl}_{M|N}\big)$ with the degree of a monomial $x\in \mathrm{U}_q\big(\mathfrak{Lgl}_{M|N}\big)$ given by its length. For $m\in\mathbb{Z}_{\geqslant 0}$, let $\mathrm{U}_{[m]}$ be the ideal of $\mathrm{U}_q\big(\mathfrak{Lgl}_{M|N}\big)$ spanned by the monomials \eqref{base:Uq} with degree $\leqslant m$.
Note that the elements
\begin{gather*}
\gamma_{ij}^{(r)}-c_{ij}^{(r)}T_{ij}^{(r)},\ \widetilde{\gamma}_{ij}^{(r)}-\widetilde{c}_{ij}^{(r)}\widetilde{T}_{ij}^{(r)}\in \mathrm{U}_{[0]}=\mathbb{C}(q)
\end{gather*}
for the constants $c_{ij}^{(r)},\ \widetilde{c}_{ij}^{(r)}\in \big\{\left(q-q^{-1}\right)^{-1},\,(q-1)^{-1}\big\}$.
This suggests that for any monomial $x_m\in \widehat{\mathfrak{B}}_2$ with degree $m\geqslant 1$, there exists an element $y_m$ contained in the PBW basis $\widehat{\mathfrak{B}}_1$ of $\mathrm{U}_q\big(\mathfrak{Lgl}_{M|N}\big)$ such that $x_m-c_my_m\in \mathrm{U}_{[m-1]}$ for $c_m\in\mathbb{C}(q)$. The images of the elements in $\widehat{\mathfrak{B}}_1$ with degree $m$ form a basis of the $\mathbb{C}(q)$-linear superspace $\mathrm{U}_{[m]}/\mathrm{U}_{[m-1]}$. Hence, the set $\widehat{\mathfrak{B}}_2$ forms a $\mathbb{C}(q)$-basis of $\mathrm{U}_q\big(\mathfrak{Lgl}_{M|N}\big)$.

Applying the same arguments as in the proof of Proposition \ref{base}, $\mathrm{U}_{\mathbb{A}}$ is spanned by the ordered monomials \eqref{base:Uq} in generators $\gamma_{ij}^{(r)},\widetilde{\gamma}_{ij}^{(r)}$ with conditions \eqref{base:1} and \eqref{base:2} as $\mathbb{A}$-linear superspace. Then, from \eqref{modulo:1}, the tensor $\mathrm{U}_{\mathbb{A}}\otimes_{\mathbb{A}}\mathbb{C}(q)$ is spanned by
the ordered monomials with the form \eqref{base:Uq} in $\gamma_{ij}^{(r)}\otimes 1,\widetilde{\gamma}_{ij}^{(r)}\otimes 1$ with conditions \eqref{base:1}--\eqref{base:3} as a $\mathbb{C}(q)$-linear superspace.
Given the $\mathbb{Z}$-gradation of $\mathrm{U}_q\big(\mathfrak{Lgl}_{M|N}\big)$ for which any components $Q_s$ is spanned by the monomials \eqref{monomial:p} of degree $s$ such that $s=\sum_{k=1}^p (r_k+1)(i_k+j_k)$. The number of these monomials with a given degree $s$ must be finite.

On the other hand, ensure the same $\mathbb{Z}$-gradation on $\mathrm{U}_{\mathbb{A}}$ with components $Q'_s=Q_s\cap \mathrm{U}_{\mathbb{A}}$. The images of $x_s\otimes 1$ for all monomials $x_s\in Q_s'$ under \eqref{iso:ua} span $Q_s$; therefore, the map \eqref{iso:ua} is bijective. Furthermore, the map \eqref{iso:ua} is $\mathbb{C}(q)$-isomorphism of superalgebras since $\mathrm{U}_{\mathbb{A}}$ is an $\mathbb{A}$-sub-superalgebra of $\mathrm{U}_q\big(\mathfrak{Lgl}_{M|N}\big)$.

\end{proof}

\begin{remark}
The proof of the proposition above refers to \cite[Chapter VIII, Section 12, Theorem 1 \& Lemma 5]{B}. More discussions for $\mathbb{A}$-form or integral form in the nonsuper case can be found in \cite{CP,HK}.
\end{remark}

The isomorphism \eqref{iso:ua} implies that the defining relations for $\mathrm{U}_{\mathbb{A}}$ can be deduced from the relations for $\mathrm{U}_q\big(\mathfrak{L}\mathfrak{gl}_{M|N}\big)$. In greater detail, the superalgebra $\mathrm{U}_{\mathbb{A}}$ admits the defining relations
\begin{align}\label{UA:1}
&\gamma_{ii}^{(0)}+\widetilde{\gamma}_{ii}^{(0)}+(q-1)\gamma_{ii}^{(0)}\widetilde{\gamma}_{ii}^{(0)}=0\ \text{for }i\in I, \\ \label{UA:2}
&\gamma_{ji}^{(0)}=\widetilde{\gamma}_{ij}^{(0)}=0\ \textrm{for}~~1\leqslant i<j\leqslant M+N,
\end{align}
and \eqref{RTT5n5}--\eqref{RTT4n5} by replacing $T_{ij}^{(r)},\widetilde{T}_{ij}^{(r)}$ with $\left(q-q^{-1}\right)\gamma_{ij}^{(r)},\left(q-q^{-1}\right)\widetilde{\gamma}_{ij}^{(r)}$ for $i,j\in I$, $r\geqslant 0$, except that, if $i=j,r=0$, replacing $T_{ii}^{(0)},\widetilde{T}_{ii}^{(0)}$ with $(q-1)\gamma_{ii}^{(0)}+1,(q-1)\widetilde{\gamma}_{ii}^{(0)}+1$.

The following theorem indicates that $\mathrm{U}_q\big(\mathfrak{Lgl}_{M|N}\big)$ can degenerate to the universal enveloping superalgebra $\mathrm{U}\big(\mathfrak{Lgl}_{M|N}\big)$ as the classical limit $q\rightarrow 1$.

\begin{theorem}\label{limit1}
There exists an isomorphism $\psi$ of superalgebras between $\mathrm{U}(\mathfrak{Lgl}_{M|N})$ to $\mathrm{U}_{\mathbb{A}}/(q-1)\mathrm{U}_{\mathbb{A}}$ such that
\begin{align*}
&E_{ij}^{(0)}\mapsto (-1)^{|i||j|}\gamma_{ij}^{(0)},\quad E_{ji}^{(0)}\mapsto -(-1)^{|i||j|}\widetilde{\gamma}_{ji}^{(0)}\quad \textrm{for}~~1\leqslant i\leqslant j\leqslant M+N, \\
&E_{ij}^{(-r)}\mapsto (-1)^{|i||j|}\gamma_{ij}^{(r)},\quad E_{ij}^{(r)}\mapsto -(-1)^{|i||j|}\widetilde{\gamma}_{ij}^{(r)}\quad\textrm{for}~~i,j\in I,\ r>0.
\end{align*}
\end{theorem}

\begin{proof}
By the defining relations of $\mathrm{U}_{\mathbb{A}}$ and modulo the ideal $(q-1)\mathrm{U}_{\mathbb{A}}$, we obtain
\begin{align}\label{UAmod:1}
&\gamma_{ii}^{(0)}\equiv -\widetilde{\gamma}_{ii}^{(0)}, \\ \label{UAmod:2}
&\left[\widetilde{\gamma}_{ii}^{(r)},\,\gamma_{ii}^{(s)}\right]\equiv \left[\gamma_{ii}^{(r)},\,\gamma_{ii}^{(s)}\right]
  \equiv \left[\widetilde{\gamma}_{ii}^{(r)},\,\widetilde{\gamma}_{ii}^{(s)}\right] \equiv 0, \\ \label{UAmod:3}
&\left[\widetilde{\gamma}_{ij}^{(r)},\,\gamma_{ij}^{(s)}\right]\equiv \left[\gamma_{ij}^{(r)},\,\gamma_{ij}^{(s)}\right]\equiv
  \left[\widetilde{\gamma}_{ij}^{(r)},\,\widetilde{\gamma}_{ij}^{(s)}\right]\equiv 0, \\ \label{UAmod:4}
&\left[\widetilde{\gamma}_{ii}^{(r)},\,\gamma_{kk}^{(s)}\right]\equiv \left[\gamma_{ii}^{(r)},\,\gamma_{kk}^{(s)}\right]\equiv
  \left[\widetilde{\gamma}_{ii}^{(r)},\,\widetilde{\gamma}_{kk}^{(s)}\right]\equiv 0, \\ \label{UAmod:5}
&\left[\widetilde{\gamma}_{ij}^{(r)},\,\gamma_{ji}^{(s)}\right]\equiv (-1)^{|i|}\left(\frac{1}{2}\right)^{\delta_{rs}}
  \left\{\delta_{s\geqslant r}\left(\gamma_{jj}^{(s-r)}-\gamma_{ii}^{(s-r)}\right) -\delta_{r\geqslant s}\left(\widetilde{\gamma}_{jj}^{(r-s)}-\widetilde{\gamma}_{ii}^{(r-s)}\right)\right\}, \\ \label{UAmod:6}
&\left[\gamma_{ij}^{(r)},\,\gamma_{ji}^{(s)}\right]\equiv (-1)^{|i|}\left(\gamma_{ii}^{(r+s)}-\gamma_{jj}^{(r+s)}\right),\quad
  \left[\widetilde{\gamma}_{ij}^{(r)},\,\widetilde{\gamma}_{ji}^{(s)}\right]\equiv
  (-1)^{|i|}\left(\widetilde{\gamma}_{jj}^{(r+s)}-\widetilde{\gamma}_{ii}^{(r+s)}\right), \\ \label{UAmod:7}
&\left[\widetilde{\gamma}_{ii}^{(r)},\,\gamma_{il}^{(s)}\right]\equiv \delta_{r\geqslant s}(-1)^{|i|}\widetilde{\gamma}_{il}^{(r-s)}-\delta_{s\geqslant r}(-1)^{|i|}\gamma_{il}^{(s-r)}, \\ \label{UAmod:8}
&\left[\gamma_{ii}^{(r)},\,\gamma_{il}^{(s)}\right]\equiv (-1)^{|i|}\gamma_{il}^{(r+s)},\quad
  \left[\widetilde{\gamma}_{ii}^{(r)},\,\widetilde{\gamma}_{il}^{(s)}\right]\equiv -(-1)^{|i|}\widetilde{\gamma}_{il}^{(r+s)}, \\ \label{UAmod:9}
&\left[\widetilde{\gamma}_{ii}^{(r)},\,\gamma_{ki}^{(s)}\right]\equiv \delta_{s\geqslant r}(-1)^{|i|}\gamma_{ki}^{(s-r)}
  -\delta_{r\geqslant s}(-1)^{|i|}\widetilde{\gamma}_{ki}^{(r-s)}, \\ \label{UAmod:10}
&\left[\gamma_{ii}^{(r)},\,\gamma_{ki}^{(s)}\right]\equiv -(-1)^{|i|}\gamma_{ki}^{(r+s)},\quad
  \left[\widetilde{\gamma}_{ii}^{(r)},\,\widetilde{\gamma}_{ki}^{(s)}\right]\equiv (-1)^{|i|}\widetilde{\gamma}_{ki}^{(r+s)}, \\ \label{UAmod:11}
&\left[\widetilde{\gamma}_{ij}^{(r)},\,\gamma_{ii}^{(s)}\right]\equiv \delta_{s\geqslant r}(-1)^{|i|}\gamma_{ij}^{(s-r)}
  -\delta_{r\geqslant s}(-1)^{|i|}\widetilde{\gamma}_{ij}^{(r-s)}, \\ \label{UAmod:12}
&\left[\gamma_{ij}^{(r)},\,\gamma_{ii}^{(s)}\right]\equiv -(-1)^{|i|}\gamma_{ij}^{(r+s)},\quad
  \left[\widetilde{\gamma}_{ij}^{(r)},\,\widetilde{\gamma}_{ii}^{(s)}\right]\equiv (-1)^{|i|}\widetilde{\gamma}_{ij}^{(r+s)},\\ \label{UAmod:13}
&\left[\widetilde{\gamma}_{ij}^{(r)},\,\gamma_{jj}^{(s)}\right]\equiv \delta_{r\geqslant s}(-1)^{|j|}\widetilde{\gamma}_{ij}^{(r-s)}
  -\delta_{s\geqslant r}(-1)^{|j|}\gamma_{ij}^{(s-r)}, \\ \label{UAmod:14}
&\left[\gamma_{ij}^{(r)},\,\gamma_{jj}^{(s)}\right]\equiv (-1)^{|j|}\gamma_{ij}^{(r+s)},\quad
  \left[\widetilde{\gamma}_{ij}^{(r)},\,\widetilde{\gamma}_{jj}^{(s)}\right]\equiv -(-1)^{|j|}\widetilde{\gamma}_{ij}^{(r+s)}, \\ \label{UAmod:15}
&\left[\widetilde{\gamma}_{ii}^{(r)},\,\gamma_{kl}^{(s)}\right]\equiv \left[\gamma_{ii}^{(r)},\,\gamma_{kl}^{(s)}\right] \equiv
  \left[\widetilde{\gamma}_{ii}^{(r)},\,\widetilde{\gamma}_{kl}^{(s)}\right]\equiv 0, \\ \label{UAmod:16}
&\left[\widetilde{\gamma}_{ij}^{(r)},\,\gamma_{kk}^{(s)}\right]\equiv \left[\gamma_{ij}^{(r)},\,\gamma_{kk}^{(s)}\right]\equiv
  \left[\widetilde{\gamma}_{ij}^{(r)},\,\widetilde{\gamma}_{kk}^{(s)}\right]\equiv 0, \\ \label{UAmod:17}
&\left[\widetilde{\gamma}_{ij}^{(r)},\,\gamma_{jl}^{(s)}\right]\equiv (-1)^{|i||j|+|i||l|+|j||l|}\left\{\delta_{r\geqslant s}
  \widetilde{\gamma}_{il}^{(r-s)}-\delta_{s\geqslant r}\gamma_{il}^{(s-r)}\right\}, \\ \label{UAmod:18}
&\left[\gamma_{ij}^{(r)},\,\gamma_{jl}^{(s)}\right]\equiv (-1)^{|i||j|+|j||l|+|i||l|}\gamma_{il}^{(r+s)},\quad
  \left[\widetilde{\gamma}_{ij}^{(r)},\,\widetilde{\gamma}_{jl}^{(s)}\right]\equiv
  -(-1)^{|i||j|+|j||l|+|i||l|}\widetilde{\gamma}_{il}^{(r+s)}, \\ \label{UAmod:19}
&\left[\widetilde{\gamma}_{ij}^{(r)},\,\gamma_{ki}^{(s)}\right]\equiv (-1)^{|i|}\left\{\delta_{s\geqslant r}\gamma_{kj}^{(s-r)}
  -\delta_{r\geqslant s}\widetilde{\gamma}_{kj}^{(r-s)}\right\}, \\ \label{UAmod:20}
&\left[\gamma_{ij}^{(r)},\,\gamma_{ki}^{(s)}\right]\equiv -(-1)^{|i|}\gamma_{kj}^{(r+s)},\quad
  \left[\widetilde{\gamma}_{ij}^{(r)},\,\widetilde{\gamma}_{ki}^{(s)}\right]\equiv (-1)^{|i|}\widetilde{\gamma}_{kj}^{(r+s)}, \\ \label{UAmod:21}
&\left[\widetilde{\gamma}_{ij}^{(r)},\,\gamma_{kl}^{(s)}\right]\equiv \left[\gamma_{ij}^{(r)},\,\gamma_{kl}^{(s)}\right]
 \equiv \left[\widetilde{\gamma}_{ij}^{(r)},\,\widetilde{\gamma}_{kl}^{(s)}\right]\equiv 0,
\end{align}
where the indexes $i,j,k,l\in I$ are pairwise distinct.

Denote for $i,j\in I$, $r\in\mathbb{Z}$ in $\mathrm{U}_{\mathbb{A}}$
\begin{equation*}
\xi_{ij}^{(r)}=\begin{cases}
(-1)^{|i||j|}\gamma_{ij}^{(r)}, &\text{if }r<0\text{ or }r=0,i\leqslant j, \\
-(-1)^{|i||j|}\widetilde{\gamma}_{ij}^{(r)}, &\text{if }r>0\text{ or }r=0,i>j.
\end{cases}
\end{equation*}
Then $\psi$ maps $E_{ij}^{(r)}$ to the image of $\xi_{i,j}^{(r)}$ in the quotient $\mathrm{U}_{\mathbb{A}}/(q-1)\mathrm{U}_{\mathbb{A}}$, that is, $\psi\left(E_{ij}^{(r)}\right)=\overline{\xi_{i,j}^{(r)}}$. Regard $\mathrm{U}_{\mathbb{A}}/(q-1)\mathrm{U}_{\mathbb{A}}$ as an (abstract) associative superalgebra, it is equipped with generators $\overline{\xi_{ij}^{(r)}}$ and follows a set of defining relations
\begin{gather*}
\left[\,\overline{\xi_{ij}^{(r)}},\,\overline{\xi_{kl}^{(s)}}\,\right]
  =\delta_{jk}\overline{\xi_{ij}^{(r)}}-(-1)^{(|i|+|j|)(|k|+|l|)}\delta_{li}\overline{\xi_{kj}^{(r)}}
\end{gather*}
for all $i,j,k,l\in I$ and $r,s\in \mathbb{Z}$,
which coincides with relation \eqref{Uni:eq}. This alignment facilitates the establishment of the isomorphism $\psi$, which is what we desire.

\end{proof}

\subsection{From quantum loop superalgebra to super Yangian}
As known, Yangians can be constructed as limit forms of the quantum loop algebras. The long-lost proofs were finally proposed by Gautam, Toledano Laredo in \cite{GT} and Guay, Ma in \cite{GM}. Similar connections for twisted Yangians and twisted quantum loop algebras had been established by Conner and Guay \cite{CG}. In this subsection, we aim to establish an isomorphism between the super Yangian $\mathrm{Y}_{\hbar}(\mathfrak{gl}_{M|N})$ and a limit form of quantum loop superalgebra $\mathrm{U}_q(\mathfrak{L}\mathfrak{gl}_{M|N})$ in terms of RTT type presentation, motivated by \cite{CG,GM}. In order to prove the isomorphism, we first introduce some notations.
Consider the following superalgebra homomorphisms sequence
\begin{gather*}
\mathrm{U}_{\mathbb{A}}\twoheadrightarrow \mathrm{U}_{\mathbb{A}}/(q-1)\mathrm{U}_{\mathbb{A}}\xlongrightarrow{\simeq} \mathrm{U}(\mathfrak{L}\mathfrak{gl}_{M|N}).
\end{gather*}
Let $\psi^{\prime}$ be the composite as above. We define the notations $\Gamma_{ij}^{(r,m)}$ and $\widetilde{\Gamma}_{ij}^{(r,m)}(r,m\geqslant 0)$ in the following way. Denote
\begin{gather*}
\Gamma_{ij}^{(0,0)}=\gamma_{ij}^{(0)},\quad \widetilde{\Gamma}_{ji}^{(0,0)}=\widetilde{\gamma}_{ji}^{(0)}\quad \textrm{if } i\leqslant j,
\end{gather*}
 and
\begin{gather*}
\Gamma_{ij}^{(0,0)}=-\widetilde{\gamma}_{ij}^{(0)},\quad \widetilde{\Gamma}_{ji}^{(0,0)}=-\gamma_{ji}^{(0)},\quad \textrm{if }i>j.
\end{gather*}
 For $r>0$, we set
\begin{gather*}
\Gamma_{ij}^{(r,0)}=\gamma_{ij}^{(r)},\quad \widetilde{\Gamma}_{ij}^{(r,0)}=\widetilde{\gamma}_{ij}^{(r)},
\end{gather*}
and recursively
\begin{gather*}
\Gamma_{ij}^{(r,m+1)}=\Gamma_{ij}^{(r+1,m)}-\Gamma_{ij}^{(r,m)},\quad \widetilde{\Gamma}_{ij}^{(r,m+1)}=\widetilde{\Gamma}_{ij}^{(r+1,m)}-\widetilde{\Gamma}_{ij}^{(r,m)}.
\end{gather*}
It is clear that $\psi^{\prime}\left(\Gamma_{ij}^{(r,m)}\right)=(-1)^m(-1)^{|i||j|} E_{ij}x^{-(r+m)}(x-1)^m$ and $\psi^{\prime}\left(\widetilde{\Gamma}_{ij}^{(r,m)}\right)=-(-1)^{|i||j|}E_{ij}x^r(x-1)^m$ for every $r\geqslant 0$.

For $-m\leqslant r<0$, define $\Gamma_{ij}^{(-m,m)}=(-1)^{m+1}\widetilde{\Gamma}_{ij}^{(0,m)}$, and recursively $\Gamma_{ij}^{(r,m+1)}=\Gamma_{ij}^{(r+1,m)}-\Gamma_{ij}^{(r,m)}$ either. It ensures that the definition of $\Gamma_{ij}^{(r,m)}$ with $r<0$ is compatible with $\psi^{\prime}\left(\Gamma_{ij}^{(r,m)}\right)$ given as before.
Introduce a filtration on $\mathrm{U}_{\mathbb{A}}$
\begin{gather*}
\textsf{F}_{[0]}\supset \textsf{F}_{[1]}\supset \textsf{F}_{[2]}\supset \ldots
\end{gather*}
by setting $\deg \Gamma_{i,j}^{(r,m)}=m$ and $\deg \left(q-q^{-1}\right)=1$ and let $\textsf{F}_{[m]}$ be
the span of all monomials of the form
\begin{gather*}
\left(q-q^{-1}\right)^{m_0}\Gamma_{i_1,j_1}^{(-r_1,m_1)}\Gamma_{i_2,j_2}^{(-r_2,m_2)}\ldots \Gamma_{i_k,j_k}^{(-r_k,m_k)},
\end{gather*}
where all $i_s,j_s\in I$, $0\leqslant r_s\leqslant \frac{1}{2}(m_s-1)$, $m_0,m_s\geqslant 0$ for $1\leqslant s\leqslant k$, $m_0+m_1+\cdots+m_k\geqslant m$.
For $r>0$, $m\geqslant 0$, we have
\begin{align*}
&\Gamma_{ij}^{(r,m)}=\Gamma_{ij}^{(0,m)}+\sum_{p=1}^{r}\binom{r}{p}\Gamma_{ij}^{(0,m+p)}\in \textsf{F}_{[m]}, \\
&\Gamma_{ij}^{(1-m,m)}=\Gamma_{ij}^{(0,m)}+\sum_{p=1}^{m-1}(-1)^p\binom{m-1}{p}\Gamma_{ij}^{(-p,m+p)} \in \textsf{F}_{[m]}, \\
&\widetilde{\Gamma}_{ij}^{(0,m)}=(-1)^{m+1}\Gamma_{ij}^{(1-m,m)}-(-1)^{m+1}\Gamma_{ij}^{(-m,m+1)} \in \textsf{F}_{[m]}, \\
&\widetilde{\Gamma}_{ij}^{(r,m)}=\widetilde{\Gamma}_{ij}^{(0,m)}+\sum_{p=1}^{r}\binom{r}{p}\widetilde{\Gamma}_{ij}^{(0,m+p)} \in \textsf{F}_{[m]}.
\end{align*}
Therefore,
\begin{gather}\label{quotient1}
(-1)^{m+1}\overline{\widetilde{\Gamma}_{ij}^{(r,m)}}=\overline{\Gamma_{ij}^{(r,m)}}= \overline{\Gamma_{ij}^{(0,m)}}\in \mathsf{F}_{[m]}/\mathsf{F}_{[m+1]}.
\end{gather}
Notice that $\textsf{F}_{[0]}=\mathrm{U}_{\mathbb{A}}$. Denote the associated $\mathbb{Z}_2$-graded algebra with respect to this filtration by $\operatorname{gr}\mathrm{U}_{\mathbb{A}}$, that is, $\operatorname{gr}\mathrm{U}_{\mathbb{A}}=\bigoplus_{m=0}^{\infty}\textsf{F}_{[m]}/\textsf{F}_{[m+1]}$.
Now we give the main result as follows.

\begin{theorem}\label{main1}
There exists an isomorphism $\varphi$ of superalgebras from the super Yangian $\mathrm{Y}_{\hbar}(\mathfrak{gl}_{M|N})$ to the filtrated algebra $\operatorname{gr}\mathrm{U}_{\mathbb{A}}$ such that $\hat{t}_{ij}^{(m+1)}\mapsto (-1)^{|i||j||+|j|} \overline{\Gamma_{ij}^{(0,m)}}$ for $i,j\in I$ and $m\geqslant 0$.
\end{theorem}
\begin{proof}
We first check that $\varphi$ is a homomorphism.
One gets from \eqref{RTT6} for $r,s\geqslant 1$,
\begin{align*}
&q_i^{-\delta_{ik}}\Gamma_{ij}^{(r,1)}\Gamma_{kl}^{(s,0)}-q_i^{\delta_{ik}}\Gamma_{ij}^{(r,0)}\Gamma_{kl}^{(s,1)}
-\left(q_i^{\delta_{ik}}-q_i^{-\delta_{ik}}\right)\Gamma_{ij}^{(r,0)}\Gamma_{kl}^{(s,0)} \\
&\hspace{4em}-\varepsilon_{ij;kl}\left\{q_j^{-\delta_{jl}}\Gamma_{kl}^{(s,0)}\Gamma_{ij}^{(r,1)}
-q_j^{\delta_{jl}}\Gamma_{kl}^{(s,1)}\Gamma_{ij}^{(r,0)}
-\left(q_j^{\delta_{jl}}-q_j^{-\delta_{jl}}\right)\Gamma_{kl}^{(s,0)}\Gamma_{ij}^{(r,0)}\right\} \\
=&\,\varepsilon_{ik;kl}\left(q_k-q_k^{-1}\right)\times\Big\{\delta_{k<i}\Gamma_{kj}^{(r,0)}\Gamma_{il}^{(s,1)}
+\delta_{i<k}\Gamma_{kj}^{(r,1)}\Gamma_{il}^{(s,0)}+\left(\delta_{k<i}+\delta_{i<k}\right)\Gamma_{kj}^{(r,0)}\Gamma_{il}^{(s,0)} \\
&\hspace{4em}-\left(\delta_{j<l}\Gamma_{kj}^{(s,1)}\Gamma_{il}^{(r,0)}+\delta_{l<j}\Gamma_{kj}^{(s,0)}\Gamma_{il}^{(r,1)}
+\left(\delta_{j<l}+\delta_{l<j}\right)\Gamma_{kj}^{(s,0)}\Gamma_{il}^{(r,0)}\right)\Big\}.
\end{align*}
Using induction on $m,n\geqslant 0$, we have
\begin{align*}
&q_i^{-\delta_{ik}}\Gamma_{ij}^{(r,m+1)}\Gamma_{kl}^{(s,n)}-q_i^{\delta_{ik}}\Gamma_{ij}^{(r,m)}\Gamma_{kl}^{(s,n+1)}
-\left(q_i^{\delta_{ik}}-q_i^{-\delta_{ik}}\right)\Gamma_{ij}^{(r,m)}\Gamma_{kl}^{(s,n)} \\
&\hspace{4em}-\varepsilon_{ij;kl}\left\{q_j^{-\delta_{jl}}\Gamma_{kl}^{(s,n)}\Gamma_{ij}^{(r,m+1)}
-q_j^{\delta_{jl}}\Gamma_{kl}^{(s,n+1)}\Gamma_{ij}^{(r,m)}
-\left(q_j^{\delta_{jl}}-q_j^{-\delta_{jl}}\right)\Gamma_{kl}^{(s,n)}\Gamma_{ij}^{(r,m)}\right\} \\
=&\,\varepsilon_{ik;kl}\left(q_k-q_k^{-1}\right)\times\Big\{\delta_{k<i}\Gamma_{kj}^{(r,m)}\Gamma_{il}^{(s,n+1)}
+\delta_{i<k}\Gamma_{kj}^{(r,m+1)}\Gamma_{il}^{(s,n)}+\left(\delta_{k<i}+\delta_{i<k}\right)\Gamma_{kj}^{(r,m)}\Gamma_{il}^{(s,n)} \\
&\hspace{4em}-\left(\delta_{j<l}\Gamma_{kj}^{(s,n+1)}\Gamma_{il}^{(r,m)}+\delta_{l<j}\Gamma_{kj}^{(s,n)}\Gamma_{il}^{(r,m+1)}
+\left(\delta_{j<l}+\delta_{l<j}\right)\Gamma_{kj}^{(s,n)}\Gamma_{il}^{(r,m)}\right)\Big\}.
\end{align*}
Taking the pair $(r,s)=(1,1)$, one gets,
\begin{align*}
&q_i^{-\delta_{ik}}\Gamma_{ij}^{(1,m+1)}\Gamma_{kl}^{(1,n)}-q_i^{\delta_{ik}}\Gamma_{ij}^{(1,m)}\Gamma_{kl}^{(1,n+1)}
-\left(q_i^{\delta_{ik}}-q_i^{-\delta_{ik}}\right)\Gamma_{ij}^{(1,m)}\Gamma_{kl}^{(1,n)} \\
&\hspace{4em}-\varepsilon_{ij;kl}\left\{q_j^{-\delta_{jl}}\Gamma_{kl}^{(1,n)}\Gamma_{ij}^{(1,m+1)}
-q_j^{\delta_{jl}}\Gamma_{kl}^{(1,n+1)}\Gamma_{ij}^{(1,m)}
-\left(q_j^{\delta_{jl}}-q_j^{-\delta_{jl}}\right)\Gamma_{kl}^{(1,n)}\Gamma_{ij}^{(1,m)}\right\} \\
=&\,\varepsilon_{ik;kl}\left(q_k-q_k^{-1}\right)\times\Big\{\delta_{k<i}\Gamma_{kj}^{(1,m)}\Gamma_{il}^{(1,n+1)}
+\delta_{i<k}\Gamma_{kj}^{(1,m+1)}\Gamma_{il}^{(1,n)}+\left(\delta_{k<i}+\delta_{i<k}\right)\Gamma_{kj}^{(1,m)}\Gamma_{il}^{(1,n)} \\
&\hspace{4em}-\left(\delta_{j<l}\Gamma_{kj}^{(1,n+1)}\Gamma_{il}^{(1,m)}+\delta_{l<j}\Gamma_{kj}^{(1,n)}\Gamma_{il}^{(1,m+1)}
+\left(\delta_{j<l}+\delta_{l<j}\right)\Gamma_{kj}^{(1,n)}\Gamma_{il}^{(1,m)}\right)\Big\}.
\end{align*}
Since both sides of this relation are included in $\mathsf{F}_{[m+n+1]}$ and $\Gamma_{ij}^{(1,m)}=\Gamma_{ij}^{(0,m+1)}+\Gamma_{ij}^{(0,m)}$, one gets for all $m,n\geqslant 0$,
\begin{align*}
&q_i^{-\delta_{ik}}\Gamma_{ij}^{(0,m+1)}\Gamma_{kl}^{(0,n)}-q_i^{\delta_{ik}}\Gamma_{ij}^{(0,m)}\Gamma_{kl}^{(0,n+1)}
-\delta_{ik}(-1)^{|i|}\left(q-q^{-1}\right)\Gamma_{ij}^{(0,m)}\Gamma_{kl}^{(0,n)} \\
&\hspace{4em}-\varepsilon_{ij;kl}\left\{q_j^{-\delta_{jl}}\Gamma_{kl}^{(0,n)}\Gamma_{ij}^{(0,m+1)}
-q_j^{\delta_{jl}}\Gamma_{kl}^{(0,n+1)}\Gamma_{ij}^{(0,m)}-\delta_{jl}(-1)^{|j|}(q-q^{-1})\Gamma_{kl}^{(0,n)}\Gamma_{ij}^{(0,m)}\right\} \\
=&\,\varepsilon_{ik;kl}(-1)^{|k|}\left(q-q^{-1}\right)\times\left\{\left(\delta_{k<i}+\delta_{i<k}\right)\Gamma_{kj}^{(0,m)}\Gamma_{il}^{(0,n)}
-\left(\delta_{j<l}+\delta_{l<j}\right)\Gamma_{kj}^{(0,n)}\Gamma_{il}^{(0,m)}\right\},
\end{align*}
modulo $\mathsf{F}_{[m+n+2]}$. That is to say, we have in $\mathsf{F}_{[m+n+1]}/\mathsf{F}_{[m+n+2]}$
\begin{align*}
&\overline{\Gamma_{ij}^{(0,m+1)}}\ \overline{\Gamma_{kl}^{(0,n)}}-\overline{\Gamma_{ij}^{(0,m)}}\ \overline{\Gamma_{kl}^{(0,n+1)}}
-\varepsilon_{ij;kl}\left(\,\overline{\Gamma_{kl}^{(0,n)}}\ \overline{\Gamma_{ij}^{(0,m+1)}}
-\overline{\Gamma_{kl}^{(0,n+1)}}\ \overline{\Gamma_{ij}^{(0,m)}}\,\right) \\
=&\,\varepsilon_{ik;kl}(-1)^{|k|}\left(\,\overline{q-q^{-1}}\,\right)\times\left(\,\overline{\Gamma_{kj}^{(0,m)}}\ \overline{\Gamma_{il}^{(0,n)}}
-\overline{\Gamma_{kj}^{(0,n)}}\ \overline{\Gamma_{il}^{(0,m)}}\,\right).
\end{align*}
It coincides with relation \eqref{Y4} in $\mathrm{Y}_{\hbar}(\mathfrak{gl}_{M|N})$ if we set $\hbar=\overline{q-q^{-1}}\in\mathsf{F}_{[1]}/\mathsf{F}_{[2]}$.

Next, we shall show that $\varphi$ is bijective.
The surjectivity of $\varphi$ follows from \eqref{quotient1}.
The rest of this proof is to verify the injectivity of $\varphi$. Set $Y=\mathrm{Y}_{\hbar}(\mathfrak{gl}_{M|N})$. Introduce the filtration
\begin{gather}\label{filtration:y}
Y\supset \hbar Y\supset \hbar^2 Y\supset \cdots
\end{gather}
on $Y$. The associated graded algebra is denoted by $\overline{Y}=\bigoplus_{s=0}^{\infty} \hbar^sY/\hbar^{s+1}Y$. Denote the image of the generators $\hat{t}_{ij}^{(m+1)}$ on $\overline{Y}$ by $\bar{t}_{ij}^{(m+1)}$. Then the ordered monomials
\begin{gather*}
\hbar^s\bar{t}_{i_1,j_1}^{(r_1)}\bar{t}_{i_2,j_2}^{(r_2)}\cdots \bar{t}_{i_t,j_t}^{(r_t)}
\end{gather*}
with the conditions
\begin{gather*}
t\geqslant 0,\quad \textrm{and}~~\bar{t}_{i_a,j_a}^{(r_a)}\prec\bar{t}_{i_{a+1},j_{a+1}}^{(r_{a+1})}(1\leqslant a\leqslant t-1)~~\textrm{if}~~|i_a|+|j_a|=|i_{a+1}|+|j_{a+1}|=\bar{1}
\end{gather*}
for all $s\geqslant 0$ constitute a basis of the component $\hbar^sY/\hbar^{s+1}Y$. There exists a natural isomorphism $\pi_{Y;s}$ of linear superspace between $Y/\hbar Y$ and $\hbar^sY/\hbar^{s+1}Y$ such that $\bar{y}\mapsto \hbar^s \bar{y}$ for any $\bar{y}\in Y/\hbar Y$.

Set $X_{\mathbb{A}}=\textrm{gr}\mathrm{U}_{\mathbb{A}}$ and $\zeta_{ij}^{(m)}=(-1)^{|i||j||+|j|}\overline{\Gamma_{ij}^{(0,m)}}$. we introduce a filtration on $X_{\mathbb{A}}$ in a similar way as \eqref{filtration:y}. Then the graded algebra is given by $\overline{X}_{\mathbb{A}}=\bigoplus_{s=0}^{\infty}\hbar^sX_{\mathbb{A}}/\hbar^{s+1}X_{\mathbb{A}}$. Denote the image of the elements $\zeta_{ij}^{(m)}$ on $\overline{X}_{\mathbb{A}}$ by $\bar{\zeta}_{ij}^{(m)}$. There also exists a natural isomorphism $\pi_{X;s}$ of linear superspace between $X_{\mathbb{A}}/\hbar X_{\mathbb{A}}$ and $\hbar^sX_{\mathbb{A}}/\hbar^{s+1}X_{\mathbb{A}}$ such that $\bar{x}\mapsto \bar{\hbar}^s \bar{x}$ for any $\bar{x}\in X_{\mathbb{A}}/\hbar X_{\mathbb{A}}$.

Take note of that $\varphi$ induces a morphism $\bar{\varphi}$ between the graded algebras. We are left to show the injectivity of $\bar{\varphi}$. Suppose that there exists a nontrivial linear combination
\begin{gather*}
\Omega_s=\sum_{\mathbf{i},\mathbf{j},\mathbf{r}} c_{\mathbf{i},\mathbf{j};s}^{\mathbf{r}}y_{\mathbf{i},\mathbf{j};s}^{\mathbf{r}}=0,
\end{gather*}
where $c_{\mathbf{i},\mathbf{j};s}^{\mathbf{r}}\in \mathbb{A}$, $y_{\mathbf{i},\mathbf{j};s}^{\mathbf{r}}=\bar{\hbar}^s \bar{t}_{i_1,j_1}^{(r_1)}\bar{t}_{i_2,j_2}^{(r_2)}\cdots \bar{t}_{i_t,j_t}^{(r_t)}$ and $\mathbf{i},\mathbf{j},\mathbf{r}$ denote the sequences $(i_1,\ldots,i_t)$, $(j_1,\ldots,j_t)$ and $(r_1,\ldots,r_t)$, respectively. The map $\psi^{\prime}$ induces an epimorphism $\bar{\psi}^{\prime}$ between $X_{\mathbb{A}}/\hbar{X}_{\mathbb{A}}$ and $\mathrm{U}\big(\mathfrak{Lg}_{M|N}\big)$ such that $\bar{\psi}^{\prime}\left(\bar{\zeta}_{ij}^{(m)}\right)=(-1)^m(-1)^{|i||j|}E_{ij}\left(1-x^{-1}\right)^m$. By Proposition \ref{U:ordered}, it is evident that the set of all ordered monomials \eqref{X:order:1} by substituting $E_{ij}\left(1-x^{-1}\right)^m$ for $x_{ij}^{r}$ with conditions \eqref{restriction2} is linearly independent, which forces all coefficients $c_{\mathbf{ij};0}^{\mathbf{r}}=0$. Then the restriction of $\bar{\varphi}$ on $Y/\hbar Y$ is injective. Moreover, the composite map $\pi_{Y;s}\circ \bar{\varphi}\circ \pi_{X;s}$ is also an injection between $\hbar^s Y/\hbar^{s+1} Y$ and $\hbar^s X_{\mathbb{A}}/\hbar^{s+1} X_{\mathbb{A}}$, proving the injectivity of $\bar{\varphi}$.

Finally, the injectivity of $\varphi$ can be obtained immediately since $\varphi$ preserves the filtration \eqref{filtration:y} on $Y$.

\end{proof}

\section{Twisted super Yangian and Twisted quantum loop superalgebra}
In this section, we focus on the case of twisted super Yangian.  We will review the definition of twisted super Yangian, more details can be found in \cite{AACFR, BR}. Let $I_+=\{1,2,\cdots,M\}$, $I_-=I\backslash I_+$, and $N=2n$. For each index $i\in I$, we  define $\theta_i$ and  $\bar{i}$ as follows.
\begin{equation*}
    \begin{aligned}
    	\theta_i=\begin{cases}
        1,&\textrm{if}~i\in I_+,\\
        (-1)^{i-M-1},&\textrm{if}~i\in I_-,
        \end{cases}
    \end{aligned}
    \qquad
    \begin{aligned}
        \bar{i}=\begin{cases}
        M+1-i,&\textrm{if}~i\in I_+,\\
        2M+2n+1-i,&\textrm{if}~i\in I_-.
        \end{cases}
    \end{aligned}
\end{equation*}

Introduce the super-transposition $\iota$\footnote{The super-transposition $\iota$ is a slight different from that transposition $t$ introduced in \cite{BR}.}
for a matrix $X=\sum_{i,j\in I}E_{ij}\otimes X_{ij}$ by
\begin{gather*}
X^{\iota}=\sum_{i,j\in I}(-1)^{|i||j|+|i|}\theta_i\theta_jE_{\bar{j}\bar{i}}\otimes X_{ij}=\sum_{i,j\in I}E_{ij}\otimes X_{ij}^{\iota},
\end{gather*}
where $$X_{ij}^{\iota}=(-1)^{|i||j|+|i|}\theta_i\theta_jX_{\bar{j}\bar{i}}.$$
Similar to the usual transposition of non-super case, the equality $(X^{\iota})^{\iota}=X$ still holds. Moreover, $(XY)^{\iota}=Y^{\iota} X^{\iota}$ holds if either $X$ or $Y$ is a constant matrix. The notation $\iota_i$ denotes the partial super-transposition by applying $\iota$ to the $i$-th tensor factor, e.g. if $X=X_1\otimes X_2$, $X^{\iota_1}=(X_1)^{\iota}\otimes X_2$ and $X^{\iota_2}=X_1\otimes (X_2)^{\iota}$.
The super-involution $\tau$ of $\mathfrak{gl}_{M|2n}$ is given by
\begin{gather*}
\tau(E_{ij})=-E_{ij}^{\iota}.
\end{gather*}
The fixed point sub-superalgebra $\mathfrak{osp}_{M|2n}$ of $\mathfrak{gl}_{M|2n}$ with respect to super-involution $\tau$ is generated by the generators $E_{ij}-(-1)^{|i||j|+|i|}\theta_i\theta_j E_{\bar{j}\bar{i}}$.
Extending $\tau$ to the polynomial current Lie superalgebra $\mathfrak{gl}_{M|N}[x]=\mathfrak{gl}_{M|N}\otimes \mathbb{C}[x]$ and the loop superalgebra $\mathfrak{L}\mathfrak{gl}_{M|N}$ by
\begin{gather*}
\tau\left(E_{ij}\otimes x^r\right)=\tau\left(E_{ij}\right)\otimes(-x)^r,\quad \textrm{for}~r\geqslant 0,\quad\text{and}\quad\tau\left(E_{ij}\otimes x^r\right)=\tau\left(E_{ij}\right)\otimes x^{-r},\quad \textrm{for}~r\in\mathbb{Z},
\end{gather*}
respectively.

\subsection{Twisted super Yangian $\mathrm{Y}^{tw}\big(\mathfrak{osp}_{M|2n}\big)$}
Similar to  twisted Yangian case, the twisted super Yangian $\mathrm{Y}^{tw}\big(\mathfrak{osp}_{M|2n}\big)$ can be defined as a sub-superalgebra of the super Yangian $\mathrm{Y}\big(\mathfrak{gl}_{M|2n}\big)$.
\begin{definition}(\cite{AACFR, BR})
Let
\begin{gather}\label{Ytw1}
\mathsf{S}(u)=\sum_{i,j\in I}\sum_{m\geq 0}E_{ij}\otimes s_{ij}^{(m)}u^{-m}, \quad\textrm{with } s_{ij}^{(0)}=\delta_{ij}
\end{gather}
be the element of $\textrm{End}(V)\otimes \mathrm{Y}\big(\mathfrak{gl}_{M|2n}\big)\big[\big[u^{-1}\big]\big]$ such that $\mathsf{S}(u)=\mathsf{T}(u)\left(\mathsf{T}(-u)\right)^{\iota}$. The twisted super Yangian $\mathrm{Y}^{tw}\big(\mathfrak{osp}_{M|2n}\big)$ is defined by the generators $s_{ij}^{(m)}$ with $i,j\in I$ and $m\geqslant 0$.
\end{definition}
The generators satisfy the following relations
\begin{align}\label{Ytw2}
\mathsf{R}^{12}(u-v)\mathsf{S}^1(u)\left(\mathsf{R}^{12}(-u-v)\right)^{\iota_1}\mathsf{S}^2(u)
&=\mathsf{S}^2(u)\left(\mathsf{R}^{12}(-u-v)\right)^{\iota_1}\mathsf{S}^1(u)\mathsf{R}^{12}(u-v), \\
\label{Ytw3}
 \left(\mathsf{S}(-u)\right)^{\iota}&=\mathsf{S}(u)+\frac{\mathsf{S}(u)-\mathsf{S}(-u)}{2u}.
\end{align}

We will refer to \eqref{Ytw2} and \eqref{Ytw3} as the \textit{quaternary} and \textit{supersymmetric} relations, respectively.
The abstract superalgebra $\mathcal{Y}$ generated by 1 and $y_{ij}^{(m)}$ for $i,j\in I, m>0$, subjected to quaternary and supersymmetric relations with respect to the matrix $Y(u)=\left(y_{ij}(u)\right)$ where $y_{ij}(u)=\delta_{ij}+\sum_{m\geqslant 1}y_{ij}^{(m)}u^{-m}$ is isomorphic to the twisted super Yangian $\mathrm{Y}^{tw}\big(\mathfrak{osp}_{M|2n}\big)$. The super Yangian
$\mathrm{Y}^{tw}\big(\mathfrak{osp}_{M|2n}\big)$ is a left coideal of $\mathrm{Y}(\mathfrak{gl}_{M|2n})$ with respect to the comultiplication $\Delta_1$ such that
\begin{align*}
\Delta_1\left(s_{i,j}^{(m)}\right)=\sum_{x,y\in I}\sum_{p=0}^{m}\sum_{a=0}^{p}(-1)^{p-a}(-1)^{|i||j|+|x||j|}\theta_j\theta_y t_{i,x}^{(a)}t_{\bar{j},\bar{y}}^{(p-a)}\otimes s_{x,y}^{(m-p)}.
\end{align*}

\begin{remark}
The map
\begin{gather*}
s_{ij}(u)\mapsto \sum_{k\in I}(-1)^{|k||j|+|k|}\operatorname{sgn}\left(\frac{2M+2n+1}{2}-k\right)\operatorname{sgn}\left(\frac{2M+2n+1}{2}-j\right) t_{ik}(u)t_{\bar{j}\bar{k}}(-u)
\end{gather*}
extends to an isomorphism between $\mathrm{Y}^{tw}\big(\mathfrak{osp}_{M|2n}\big)$ and the superalgebra $Y(M|2n)^+$ defined in \cite{BR}.
\end{remark}

We can also define the superalgebra $\mathrm{Y}^{tw}\big(\mathfrak{osp}_{M|2n}\big)$ by taking a deformation of generators related to $\hbar\in\mathbb{C}\backslash\{0\}$: $\hat{s}_{ij}^{(m)}=\hbar^{1-m}s_{ij}^{(m)}$, denoted by $\mathrm{Y}^{tw}_{\hbar}(\mathfrak{osp}_{M|2n})$. Moreover, it degenerates into $\mathrm{U}\big(\mathfrak{osp}_{M|2n}[x]\big)$ as $\hbar\mapsto 0$. As a sub-superalgebra of $\mathrm{Y}_{\hbar}\big(\mathfrak{gl}_{M|2n}\big)$, one gets that
\begin{gather}\label{Ytw4}
\hat{s}_{ij}^{(m)}=\hbar\sum_{k\in I}\sum_{p=1}^{m-1}
(-1)^{p}(-1)^{|i||j|+|j||k|}\theta_k\theta_j\hat{t}_{ik}^{(m-p)}\hat{t}_{\bar{j}\bar{k}}^{(p)}+(-1)^m\theta_i\theta_j\hat{t}_{\bar{j}\bar{i}}^{(m)}
+(-1)^{|i||j|+|j|}\hat{t}_{ij}^{(m)},
\end{gather}
for $i,j\in I$ and $m\geqslant0$.

\subsection{Twisted quantum loop superalgebra $\mathrm{U}_q^{tw}$}
In $\mathrm{U}_q\big(\mathfrak{L}\mathfrak{gl}_{M|2n}\big)$, we define
\begin{gather*}
S(u)=T(u)\left(\widetilde{T}\left(u^{-1}\right)\right)^{\iota}=\sum_{i,j\in I}E_{ij}\otimes S_{ij}(u)\quad\textrm{with }S_{ij}(u)=\sum\limits_{r\geq 0}S_{ij}^{(r)}u^r.
\end{gather*}

\begin{definition}
The twisted quantum loop superalgebra $\mathrm{U}_q^{tw}$ is an associative superalebra generated by the coefficients $S_{ij}^{(r)}$ of the matrix elements $S_{ij}(u)$ for $i,j\in I$.
\end{definition}
It is clear that
\begin{align}\label{Utw1}
&S_{ij}^{(r)}=\sum_{k\in I}\sum_{p=0}^{r}(-1)^{|i||j|+|i||k|}\theta_k\theta_j
T_{ik}^{(r-p)}\widetilde{T}_{\bar{j}\bar{k}}^{(p)}, \\ \label{Utw2}
&S_{ij}^{(0)}=0,\quad\textrm{if}~~j\in I_+,i\in I_-.
\end{align}

Set $Q=P^{\iota_1}$. Note that $PQ=QP=Q$, and $P^{\iota_1\iota_2}=P$.
\begin{lemma}\label{eqs}
The following relations hold
\begin{align} \label{eq1}
&\Re_q(u,v)=uv\Re_q(v^{-1},u^{-1}), \\ \label{eq2}
&\Re_q(u,v)\left(\Re_q(u^{-1},v)\right)^{\iota_1}=\left(\Re_q(u^{-1},v)\right)^{\iota_1}\Re_q(u,v), \\ \label{eq3}
&\left(\widetilde{T}^1(u)\right)^{\iota}\left(\Re_q^{12}(u,v)\right)^{\iota_1}T^2(v)
=T^2(v)\left(\Re_q^{12}(u,v)\right)^{\iota_1}\left(\widetilde{T}^1(u)\right)^{\iota}, \\ \label{eq4}
&\Re_q^{12}(u,v)\left(\widetilde{T}^1(u^{-1})\right)^{\iota}\left(\widetilde{T}^2(v^{-1})\right)^{\iota}
=\left(\widetilde{T}^2(v^{-1})\right)^{\iota}\left(\widetilde{T}^1(u^{-1})\right)^{\iota}\Re_q^{12}(u,v).
\end{align}

\end{lemma}
\begin{proof}
Relations \eqref{eq1} and \eqref{eq2} can be checked by straightforward calculations. Applying the partial transposition $\iota_1$ to both sides of \eqref{RTT5}, we get relation \eqref{eq3}. For relation \eqref{eq4}, we define operator $\sigma$ given by $\sigma(f)=P^{12}fP^{12}$ for $f\in\textrm{End}V^{\otimes 2}\otimes \mathrm{U}_q\big(\mathfrak{L}\mathfrak{gl}_{M|2n}\big)\big[\big[u^{\pm 1},v^{\pm 1}\big]\big]$. We apply the operator $\sigma$ on relation \eqref{RTT4} to get
\begin{gather}\label{eq5}
P\Re_q^{12}(u,v)P\widetilde{T}^2(u)\widetilde{T}^1(v)
=\widetilde{T}^1(v)\widetilde{T}^2(u)P\Re_q^{12}(u,v)P.
\end{gather}
We then use the super-transposition $\iota_2$ on \eqref{eq5}:
\begin{gather}\label{eq6}
\left(\widetilde{T}^2(u)\right)^{\iota}Q\left(\Re_q^{12}(u,v)\right)^{\iota_2}Q\widetilde{T}^1(v)
=\widetilde{T}^1(v)Q\left(\Re_q^{12}(u,v)\right)^{\iota_2}Q\left(\widetilde{T}^2(u)\right)^{\iota}.
\end{gather}
Applying $\sigma$ and $\iota_2$ once again to \eqref{eq6} and $PQ=QP=Q$, we obtain
\begin{gather}\label{eq7}
P\Re_q^{12}(u,v)P\left(\widetilde{T}^2(v)\right)^{\iota}\left(\widetilde{T}^1(u)\right)^{\iota}
=\left(\widetilde{T}^1(u)\right)^{\iota}\left(\widetilde{T}^2(v)\right)^{\iota}P\Re_q^{12}(u,v)P.
\end{gather}
Finally, using $\sigma$ on \eqref{eq7} and replacing the pair $(u,v)$ with $(v^{-1},u^{-1})$, we find that relation \eqref{eq4} can be deduced from $\eqref{eq1}$.

\end{proof}

Lemma \ref{eqs} implies that the matrix $S(u)$ satisfies the quaternary relation,
\begin{gather}\label{Utw3}
\Re_q^{12}(u,v)S^1(u)\left(\Re_q^{12}\left(u^{-1},v\right)\right)^{\iota_1}S^2(v)
=S^2(v)\left(\Re_q^{12}\left(u^{-1},v\right)\right)^{\iota_1}S^1(u)\Re_q^{12}(u,v).
\end{gather}

\begin{remark}
Let $O=\sum\limits_{k\in I}E_{k\bar{k}}$ and  $G=\sum\limits_{k\in I}\left(E_{2k-1,2k}+qE_{2k,2k-1}\right)$.
Then the subalgebras generated by the coefficients of the matrix elements of the matrix $S'(u)=(O\otimes 1)T(u)(O\otimes 1)\left(\widetilde{T}\left(u^{-1}\right)\right)^{\iota}$ and $S''(u)=(O\otimes 1)T(u)(OG\otimes 1)\left(\widetilde{T}\left(u^{-1}\right)\right)^{\iota}$ with the defining relations are isomorphic to $\mathrm{Y}_q^{tw}\big(\mathfrak{o}_M\big)$ if $n=0$ and $\mathrm{Y}_q^{tw}\big(\mathfrak{sp}_{2n}\big)$ if $M=0$, respectively. Please see the definition of twisted $q$-Yangians $\mathrm{Y}_q^{tw}\big(\mathfrak{o}_M\big)$ and $\mathrm{Y}_q^{tw}\big(\mathfrak{sp}_{2n}\big)$ in \cite{MRS}.
\end{remark}
\begin{remark}
The twisted quantum loop superalgebra $\mathrm{U}_q^{tw}$ is a left coideal of $\mathrm{U}_q\big(\mathfrak{L}\mathfrak{gl}_{M|2n}\big)$ such that
\begin{gather*}
\Delta_2\left(S_{ij}^{(r)}\right)=\sum_{x,y\in I}\sum_{p=0}^r\sum_{a=0}^{p}(-1)^{(|i|+|x|)(|j|+|x|)+|y|(|j|+|y|)}\theta_y\theta_j
T_{ix}^{(a)}\widetilde{T}_{\bar{j}\bar{y}}^{(p-a)}\otimes S_{xy}^{(r-p)}.
\end{gather*}
\end{remark}

The next proposition is essential as it establishes a weak PBW theorem for twisted quantum loop superalgebra $\mathrm{U}_q^{tw}$.
\begin{proposition}\label{order}
Let $\mathbf{B}_q$ be the associative superalgebra generated by elements $B_{ij}^{(r)}$ for $i,j\in I$, $r\in\mathbb{Z}_{\geqslant 0}$ and with the defining relation
\begin{gather}\label{B:order:0}
\Re_q^{12}(u,v)B^1(u)\left(\Re_q^{12}\left(u^{-1},v\right)\right)^{\iota_1}B^2(v)
=B^2(v)\left(\Re_q^{12}\left(u^{-1},v\right)\right)^{\iota_1}B^1(u)\Re_q^{12}(u,v),
\end{gather}
where
\begin{gather*}
B(u)=\sum_{i,j\in I}E_{ij}\otimes B_{ij}(u)~~\textrm{with}~~B_{ij}(u)=\sum_{r\geq 0}B_{ij}^{(r)}u^r.
\end{gather*}
Define an order on the generators: $B_{i_1,j_1}^{(r_1)}\preceq B_{i_2,j_2}^{(r_2)}$ if and only if $(i_1,j_1;r_1)\preceq(i_2,j_2;r_2)$. Then the superalgebra $\mathbf{B}_q$ is spanned by the ordered monomials
\begin{gather}\label{B:order:1}
B_{i_1,j_1}^{(r_1)}B_{i_2,j_2}^{(r_2)}\cdots B_{i_t,j_t}^{(r_t)}
\end{gather}
with conditions
\begin{gather}\label{restriction4}
t\geqslant 0,\quad \textrm{and}~~B_{i_a,j_a}^{(r_a)}\prec B_{i_{a+1},j_{a+1}}^{(r_{a+1})}(1\leqslant a\leqslant t-1)~~\textrm{if}~~|i_a|+|j_a|=|i_{a+1}|+|j_{a+1}|=\bar{1}.
\end{gather}
In particular, this statement still holds for $\mathrm{U}_q^{tw}$.
\end{proposition}

\begin{proof}
Denote that
\begin{align*}
&\alpha_{ijkl}(u,v)=(-1)^{|i||j|+|i||k|+|j||k|}\left(q_k^{\delta_{kj}}-uvq_k^{-\delta_{kj}}\right)B_{i\bar{j}}(u)B_{k\bar{l}}(v) \\
&\qquad\qquad\qquad+(-1)^{(|j|+|k|)|l|}\theta_j\theta_k\left(q-q^{-1}\right)\left(\delta_{k<j}+\delta_{k>j}uv\right)B_{i\bar{k}}(u)B_{j\bar{l}}(v), \\
&\beta_{ijkl}(v,u)=(-1)^{|k||l|+|j||k|+|j||l|}\left(q_k^{\delta_{kj}}-uvq_k^{-\delta_{kj}}\right)B_{i\bar{j}}(v)B_{k\bar{l}}(u) \\
&\qquad\qquad\qquad+(-1)^{(|j|+|k|)|i|}\theta_j\theta_k\left(q-q^{-1}\right)\left(\delta_{\bar{k}<\bar{j}}+\delta_{\bar{k}>\bar{j}}uv\right)B_{i\bar{k}}(v)B_{j\bar{l}}(u).
\end{align*}
Then we can rewrite \eqref{B:order:0} as
\begin{equation}\label{B:order:2}
\begin{split}
&\left(uq_i^{\delta_{ik}}-vq_i^{-\delta_{ik}}\right)\alpha_{ijkl}(u,v)
+\varepsilon_{ik;kj}\left(q_k-q_k^{-1}\right)\left(\delta_{k<i}u+\delta_{k>i}v\right)\alpha_{kjil}(u,v) \\
=&\left(uq_j^{\delta_{jl}}-vq_j^{-\delta_{jl}}\right)\beta_{klij}(v,u)
+\varepsilon_{kl;lj}\left(q_l-q_l^{-1}\right)\left(\delta_{\bar{j}<\bar{l}}u+\delta_{\bar{j}>\bar{l}}v\right)\beta_{kjil}(v,u).
\end{split}
\end{equation}
Introduce a filtration on the superalgebra $\mathbf{B}_q$ and set $\deg B_{ij}^{(r)}=r$. It suffices to prove that the conclusion of this lemma still holds for the graded algebra $\mathbf{B}'_q:=gr\mathbf{B}_q$ with respect to this filtration. For convenience, we still denote the image $\overline{B_{ij}^{(r)}}$ in $\mathbf{B}'_q$ by $B_{ij}^{(r)}$, and we also denote $B_{ij}(u)$ in the same way. The notations $\alpha_{ijkl}(u,v)$ and $\beta_{ijkl}(v,u)$ in relation \eqref{B:order:2} can be replaced in $\mathbf{B}'_q$ with
\begin{equation*}
\begin{split}
&\alpha_{ijkl}(u,v)=(-1)^{|i||j|+|i||k|+|j||k|}q_k^{-\delta_{kj}}B_{i\bar{j}}(u)B_{k\bar{l}}(v) \\
&\qquad\qquad\qquad+(-1)^{(|j|+|k|)|l|}\theta_j\theta_k\left(q-q^{-1}\right)\delta_{k>j}B_{i\bar{k}}(u)B_{j\bar{l}}(v), \\
&\beta_{ijkl}(v,u)=(-1)^{|k||l|+|j||k|+|j||l|}q_k^{-\delta_{kj}}B_{i\bar{j}}(v)B_{k\bar{l}}(u) \\
&\qquad\qquad\qquad+(-1)^{(|j|+|k|)|i|}\theta_j\theta_k\left(q-q^{-1}\right)\delta_{\bar{k}>\bar{j}}B_{i\bar{k}}(v)B_{j\bar{l}}(u).
\end{split}
\end{equation*}

Furthermore, introduce the filtration on $\mathbf{B}'_q$ with $\deg B_{ij}^{(r)}=i$. We shall utilize the graded algebra $gr\mathbf{B}'_q$ subject to this filtration defined by
\begin{align*}
    gr\mathbf{B}'_q=\bigoplus_{p\geqslant 0}\mathsf{B}'_{[p+1]}/\mathsf{B}'_{[p]}\ \textrm{with}\ \mathsf{B}'_{[p]}:=\left\{X\in \mathbf{B}'_q|\deg (X)\leqslant p\right\}.
\end{align*}
Relation \eqref{B:order:2} for $\mathbf{B}'_q$ has the form
\begin{equation}\label{B:order:3}
\begin{split}
&(-1)^{|i||j|+|i||k|+|j||k|}q_k^{-\delta_{kj}}\left(uq_i^{\delta_{ik}}-vq_i^{-\delta_{ik}}\right)B_{i\bar{j}}(u)B_{k\bar{l}}(v)+Q_1 \\
&+\left(q-q^{-1}\right)q_i^{-\delta_{ij}}\left(\delta_{k<i}u+\delta_{k>i}v\right)B_{k\bar{j}}(u)B_{i\bar{l}}(v)+Q_2 \\
=&(-1)^{|i||j|+|i||l|+|j||l|}q_i^{-\delta_{il}}\left(uq_j^{\delta_{jl}}-vq_j^{-\delta_{jl}}\right)B_{k\bar{l}}(v)B_{i\bar{j}}(u)+Q_3 \\
&+\varepsilon_{ik;jl}\left(q-q^{-1}\right)q_i^{-\delta_{ij}}\left(\delta_{\bar{j}<\bar{l}}u+\delta_{\bar{j}>\bar{l}}v\right)B_{k\bar{j}}(v)B_{i\bar{l}}(u)+Q_4,
\end{split}
\end{equation}
where
\begin{align*}
&Q_1=-(-1)^{(|j|+|k|)|l|}\theta_j\theta_k\left(q-q^{-1}\right)\left(uq_i^{\delta_{ik}}-vq_i^{-\delta_{ik}}\right)
\delta_{k>j}B_{i\bar{k}}(u)B_{j\bar{l}}(v), \\
&Q_2=-(-1)^{|i||j|}\varepsilon_{ij;kl}\theta_i\theta_j\left(q-q^{-1}\right)^2\left(\delta_{k<i}u+\delta_{k>i}v\right)
\delta_{i>j}B_{k\bar{i}}(u)B_{j\bar{l}}(v), \\
&Q_3=-(-1)^{|i||k|+|k||l|}\theta_i\theta_l\left(q-q^{-1}\right)\left(uq_j^{\delta_{jl}}-vq_j^{-\delta_{jl}}\right)
\delta_{\bar{i}>\bar{l}}B_{k\bar{i}}(v)B_{l\bar{j}}(u), \\
&Q_4=-(-1)^{|i||k|+|k||l|+|j||l|}\theta_i\theta_j\left(q-q^{-1}\right)^2\left(\delta_{\bar{j}<\bar{l}}u+\delta_{\bar{j}>\bar{l}}v\right)
\delta_{\bar{i}>\bar{j}}B_{k\bar{i}}(v)B_{j\bar{l}}(u).
\end{align*}

Now we consider the value of $\delta_{\bar{a}>\bar{b}}$.
Since the subindex $\bar{a}>\bar{b}$ for $a,b\in I$ is not equivalent to $a>b$ or $a<b$ exactly,
we need to work with the monomial $B_{i\bar{j}}(u)B_{k\bar{l}}(v)$ for the following four cases: (C1) $i\in I_{\pm}$, $j,l\in I_{\mp}$; (C2) $i,j,l\in I_{\pm}$ simultaneously; (C3) $i,j\in I_{\pm}$, $l\in I_{\mp}$; (C4) $i,l\in I_{\pm}$, $j\in I_{\mp}$, where $B_{i\bar{j}}(u)$ does not precede $B_{k\bar{l}}(v)$. We say that the formal power series $B_{i\bar{j}}(u)B_{k\bar{l}}(v)=\sum B_{i\bar{j}}^{(r)}B_{k\bar{l}}^{(s)} u^rv^s$ satisfies the condition "$\textbf{LR}$" (resp. "$\textbf{BLR}$") in $\mathbf{B}_q'$ if for all $r,s\geqslant 0$, the monomial $B_{i\bar{j}}^{(r)}B_{k\bar{l}}^{(s)}$ can be linearly represented by the monomials $B_{i_1,j_1}^{(r_1)}B_{i_2,j_2}^{(r+s-r_1)}$ with degree $r+s$ and $i_1<i_2$ (resp. $(i_1,j_1;0)\preceq (i_2,j_2;0)$).

{\bf Case C1}. In this case, the terms $Q_1$--$Q_4$ are included in $\mathsf{B}'_{[k+i-1]}$. It is enough to show that the images of the ordered monomails \eqref{B:order:1} in $gr\mathbf{B}'_q$ span the algebra. The corresponding notations $\alpha_{ijkl}(u,v)$, $\beta_{ijkl}(v,u)$ in relation \eqref{B:order:2} can be replaced in $gr\mathbf{B}'_q$ by
\begin{align*}
&\alpha_{ijkl}(u,v)=(-1)^{|i||j|+|i||k|+|j||k|}q_j^{-\delta_{jl}}B_{i\bar{j}}(u)B_{k\bar{l}}(v), \\
&\beta_{ijkl}(v,u)=(-1)^{|k||l|+|j||k|+|j||l|}q_j^{-\delta_{jl}}B_{i\bar{j}}(v)B_{k\bar{l}}(u).
\end{align*}
Then we have for $i>k$
\begin{align*}
&q_k^{-\delta_{kj}}B_{i\bar{j}}(u)B_{k\bar{l}}(v)
=\varepsilon_{ij;kl}\frac{uq_j^{\delta_{jl}}-vq_j^{-\delta_{jl}}}{u-v}q_i^{-\delta_{il}}B_{k\bar{l}}(v)B_{i\bar{j}}(u) \\
&\qquad-\varepsilon_{ki;ij}\frac{q_i-q_i^{-1}}{u-v}q_i^{-\delta_{ij}}\left\{uB_{k\bar{j}}(u)B_{i\bar{l}}(v)
+\varepsilon_{ik;jl}(\delta_{\bar{j}<\bar{l}}u+\delta_{\bar{j}>\bar{l}}v)B_{k\bar{j}}(v)B_{i\bar{l}}(u)\right\}.
\end{align*}
Take $i=k$ and $\bar{j}<\bar{l}$, we obtain
\begin{align*}
q_i^{-\delta_{il}}B_{i\bar{l}}(v)B_{i\bar{j}}(u)=-\frac{uq-vq^{-1}}{u-v}q_i^{-\delta_{ij}}B_{i\bar{j}}(u)B_{i\bar{l}}(v)
+\frac{(q_i-q_i^{-1})u}{u-v}q_i^{-\delta_{ij}}B_{i\bar{j}}(v)B_{i\bar{l}}(u).
\end{align*}
It means that $B_{i\bar{j}}(u)B_{k\bar{l}}(v)$ with $|i|\neq |j|$, $|i|\neq |l|$ satisfies the condition "$\textbf{BLR}$".
Next, take $i=k$ and $j=l$ simultaneously, we have in $gr\mathbf{B}'_q$
\begin{gather}\label{B:order:4}
B_{i\bar{j}}(u)B_{i\bar{j}}(v)+B_{i\bar{j}}(v)B_{i\bar{j}}(u)=0.
\end{gather}
This implies that $\left[B_{i\bar{j}}^{(r)},\,B_{i\bar{j}}^{(s)}\right]=0$ for all $r,s\geqslant 0$.

{\bf Case C2}. In this case, $|i|=|j|=|l|$. Then the subindexes $\bar{i}>\bar{j}$ means $i<j$, and $\bar{i}>\bar{l}$ means $i<l$. Recognize that if we set $i>k$, the condition $\textbf{LR}$ holds for the case of $i\leqslant j$ since $k<j$ and $k<l$. If $i>j$ and $i\geqslant l$, we can deduce the same arguments as {\bf Case C1} in $gr\mathbf{B}'_q$. If $l>i>j$, the case $|k|\neq |i|$ holds clearly. As for $|k|=|i|$, take $i>k$, then we have
\begin{align*}
&q_k^{-\delta_{kj}}B_{i\bar{j}}(u)B_{k\bar{l}}(v)=B_{k\bar{l}}(v)B_{i\bar{j}}(u)
+\frac{q_i-q_i^{-1}}{u-v}q_i^{-\delta_{ij}}\left\{vB_{k\bar{j}}(v)B_{i\bar{l}}(u)-uB_{k\bar{j}}(u)B_{i\bar{l}}(v)\right\} +\theta_i\theta_j\\
&\times \frac{\left(q-q^{-1}\right)^2u}{u-v}B_{k\bar{i}}(u)B_{j\bar{l}}(v)-\theta_i\theta_l\left(q_i-q_i^{-1}\right)B_{k\bar{i}}(v)B_{l\bar{j}}(u) +\theta_j\theta_k(q_i-q_i^{-1})\delta_{k>j}B_{i\bar{k}}(u)B_{j\bar{l}}(v).
\end{align*}
Here $k\leqslant j$ means that the last term in the right side vanish, thus the condition "$\textbf{LR}$" holds for $B_{i\bar{j}}(u)B_{k\bar{l}}(v)$. But for $k>j$, it still holds since $B_{i\bar{k}}(u)B_{j\bar{l}}(v)$ satisfies the condition "$\textbf{LR}$".

Taking $i=k$, we have for $\bar{j}>\bar{l}$
\begin{align*}
B_{i\bar{j}}(v)B_{i\bar{l}}(u)=&\frac{uq-vq^{-1}}{u-v}B_{i\bar{l}}(u)B_{i\bar{j}}(v)
-\frac{\left(q_i-q_i^{-1}\right)u}{u-v}B_{i\bar{l}}(v)B_{i\bar{j}}(u) \\
&+\theta_i\theta_l\left(q_i-q_i^{-1}\right)B_{i\bar{i}}(v)B_{j\bar{l}}(u)
-\theta_i\theta_l\left(q_i-q_i^{-1}\right)\frac{uq-vq^{-1}}{u-v}B_{i\bar{i}}(u)B_{l\bar{j}}(v),
\end{align*}
where $B_{i\bar{i}}(v)B_{j\bar{l}}(u)$ satisfies the condition "$\textbf{BLR}$".

Now let $i=k$ and $j=l$ again, then we have in the whole $\mathbf{B}'_q$
\begin{gather*}
B_{i\bar{j}}(u)B_{i\bar{j}}(v)=B_{i\bar{j}}(v)B_{i\bar{j}}(u)+q_i^{\delta_{ij}}\theta_i\theta_j\left(q_i-q_i^{-1}\right)
\left\{\delta_{i>j}B_{i\bar{i}}(u)B_{j\bar{j}}(v)-\delta_{i<j}B_{i\bar{i}}(v)B_{j\bar{j}}(u)\right\}.
\end{gather*}
If $i\geqslant j$, we also have \eqref{B:order:4} in $gr\mathbf{B}'_q$. Otherwise, for $r>s$, we have
\begin{gather*}
B_{i\bar{j}}^{(r)}B_{i\bar{j}}^{(s)}=B_{i\bar{j}}^{(s)}B_{i\bar{j}}^{(r)}
+q_i^{\delta_{ij}}\theta_i\theta_j\left(q_i-q_i^{-1}\right)B_{i\bar{i}}^{(s)}B_{j\bar{j}}^{(r)},
\end{gather*}
where $(i,\bar{i};s)\preceq(j,\bar{j};r)$ due to $i<j$.

{\bf Case C3}. In this case, $|i|=|j|$ and $|i|\neq |l|$. Then $\bar{i}>\bar{j}$ means $i<j$ but $\bar{i}>\bar{l}$ means $i>l$. If $i\geqslant j$, the same statement can be proved as {\bf Case C1}. If $i<j$, take $i>k$, then one has
\begin{align*}
B_{i\bar{j}}(u)B_{k\bar{l}}(v)=&\frac{uq_j^{\delta_{jl}}-vq_j^{-\delta_{jl}}}{u-v}q_i^{-\delta_{il}}B_{k\bar{l}}(v)B_{i\bar{j}}(u) -\frac{\left(q_i-q_i^{-1}\right)u}{u-v}B_{k\bar{j}}(u)B_{i\bar{l}}(v) \\
&+\frac{q_k-q_k^{-1}}{u-v}\left(\delta_{\bar{j}<\bar{l}}u+\delta_{\bar{j}>\bar{l}}v\right)\left\{B_{k\bar{j}}(v)B_{i\bar{l}}(u)
       -(-1)^{|i|}\theta_i\theta_j(q-q^{-1})B_{k\bar{i}}(v)B_{j\bar{l}}(u)\right\} \\
&-(-1)^{|i|+|k|}\theta_i\theta_l\frac{q-q^{-1}}{u-v}\left(uq_j^{\delta_{jl}}-vq_j^{-\delta_{jl}}\right)\delta_{i>l}B_{k\bar{i}}(v)B_{l\bar{j}}(u).
\end{align*}
It satisfies the condition "$\textbf{LR}$" for $i\leqslant l$ since the last term of this equation vanishes. While $i>l$, the monomial $B_{k\bar{i}}(v)B_{l\bar{j}}(u)$ holds the condition "$\textbf{LR}$" from {\bf Case C1} for $|k|\neq |j|$ or {\bf Case C2} for $|k|=|j|$.

Taking $i=k$, one gets for $\bar{l}>\bar{j}$
\begin{align*}
q_i^{-\delta_{il}}B_{i\bar{l}}(v)B_{i\bar{j}}(u)=&\frac{uq-vq^{-1}}{u-v}B_{i\bar{j}}(u)B_{i\bar{l}}(v)+\theta_i\theta_l\left(q-q^{-1}\right)\delta_{i>l}B_{i\bar{i}}(v)B_{l\bar{j}}(u) \\
&+\frac{\left(q-q^{-1}\right)u}{u-v}\left\{\theta_i\theta_j\left(q-q^{-1}\right)B_{i\bar{i}}(v)B_{j\bar{l}}(u)
      -(-1)^{|i|}B_{i\bar{j}}(v)B_{i\bar{l}}(u)\right\},
\end{align*}
where $B_{i\bar{i}}(v)B_{l\bar{j}}(u)$ satisfies the condition "$\textbf{BLR}$" from {\bf Case C2}. We can disregard the case of $j=l$ since, in this situation, it always holds that $j\neq l$.

{\bf Case C4}. In this case, $|i|=|l|$ and $|i|\neq |j|$. Then $\bar{i}>\bar{l}$ means $i<l$ but $\bar{i}>\bar{j}$ means $i>j$. So we can deduce it from {\bf Case C1}--{\bf Case C3}.

Therefore, we complete the proof.

\end{proof}


\subsection{From twisted quantum loop superalgebra to twisted super Yangian}
This subsection is devote to establish the relationship between the twisted quantum loop superalgebra and twisted super Yangian of orthosymmplectic Lie superalgebra. Indeed, we will prove that the restriction of $\varphi$, in Theorem \ref{main1} to the twisted super Yangian $\mathrm{Y}^{tw}_{\hbar}\big(\mathfrak{osp}_{M|2n}\big)$ is isomorphic to the graded algebra $\overline{V}^{tw}$ derived from $\mathrm{U}_{\mathbb{A}}^{tw}=\mathrm{U}_q^{tw}\cap \mathrm{U}_{\mathbb{A}}$.

We proceed our work to introduce the elements $\lambda_{ij}^{(r)}$ for $i,j\in I$, $r\geqslant 0$: $\lambda_{ij}^{(r)}=\displaystyle\frac{S_{ij}^{(r)}}{q-q^{-1}}$ except for $r=0$, $i=j$, set $\lambda_{ii}^{(0)}=\displaystyle{\frac{ S_{ii}^{(0)}-1}{q-1}}$. Using Proposition \ref{order}, the ordered monomials with the form \eqref{B:order:1} in elements $\lambda_{ij}^{(r)}$ also span the $\mathbb{C}(q)$-linear superspace $\mathrm{U}_q^{tw}$. All elements $\lambda_{ij}^{(r)}$ are included in $\mathrm{U}_{\mathbb{A}}$ and hence, generate $\mathrm{U}_{\mathbb{A}}^{tw}$. Note that module $(q-1)\mathrm{U}_{\mathbb{A}}$,
the elements $\lambda_{ij}^{(r)}$ satisfy
\begin{gather*}
\lambda_{ij}^{(r)}=\gamma_{ij}^{(r)}+(-1)^{|i||j|+|i|}\theta_i\theta_j \widetilde{\gamma}_{\bar{j}\bar{i}}^{(r)}
\end{gather*}
owing to \eqref{Utw1}.

Furthermore, denote $\Lambda_{ij}^{(r,m)}$ for $r,m\geqslant 0$ in $\mathrm{U}_{\mathbb{A}}^{tw}$.
Set $\Lambda_{ij}^{(r,0)}=\lambda_{ij}^{(r)}$ for $r>0$ except if $r=0$, $i\in I_-$ and  $j\in I_+$, set $\Lambda_{ij}^{(0,0)}=-(-1)^{|i||j|+|i|}\theta_i\theta_j\lambda_{\bar{j}\bar{i}}^{(0)}$; and set recursively $\Lambda_{ij}^{(r,m+1)}=\Lambda_{ij}^{(r+1,m)}-\Lambda_{ij}^{(r,m)}$ for $m\geqslant 0$. Then for $r\geqslant 1$,
\begin{align*}
\Lambda_{ij}^{(r,m)}=&\left(q-q^{-1}\right)\sum_{k\in I}\sum_{p=1}^{r-1}(-1)^{|i||j|+|i||k|}\theta_k\theta_j
\Gamma_{ik}^{(r-p,m)}\widetilde{\Gamma}_{\bar{j}\bar{k}}^{(p,0)}  \\
&+\left(q-q^{-1}\right)\sum_{k\in I}\sum_{p=1}^m(-1)^{|i||j|+|i||k|}\theta_k\theta_j
\Gamma_{ik}^{(1,p-1)}\widetilde{\Gamma}_{\bar{j}\bar{k}}^{(r,m-p)}  \\
&+\left(q-q^{-1}\right)\sum_{k\geq i}(-1)^{|i||j|+|i||k|}\theta_k\theta_j\left(\frac{1}{1+q^{-1}}\right)^{\delta_{ik}}
\Gamma_{ik}^{(0,0)}\widetilde{\Gamma}_{\bar{j}\bar{k}}^{(r,m)} \\
&+\left(q-q^{-1}\right)\sum_{\bar{k}\leq\bar{j}}(-1)^{|i||j|+|i||k|}\theta_k\theta_j\left(\frac{1}{1+q^{-1}}\right)^{\delta_{jk}}
\Gamma_{ik}^{(r,m)}\widetilde{\Gamma}_{\bar{j}\bar{k}}^{(0,0)} \\
&+\sum_{k\in I}(-1)^{|i||j|+|i||k|}\theta_k\theta_j\left\{\delta_{ik}\widetilde{\Gamma}_{\bar{j}\bar{k}}^{(r,m)}
+\delta_{jk}\Gamma_{ik}^{(r,m)}\right\}.
\end{align*}
The elements $\Lambda_{i,j}^{(r,m)}$ are included in $\mathsf{F}_{[m]}$ for every $m\geqslant 0$ and the first, third and forth terms vanish in $\mathsf{F}_{[m]}/\mathsf{F}_{[m+1]}$. The following lemma is a straightforward consequence of the statement as above.

\begin{lemma}\label{LM3}
It holds for $r\geqslant 1$ and $m\geqslant 0$,
\begin{align*}
\Lambda_{ij}^{(r,m)}&\equiv\left(q-q^{-1}\right)\sum_{k\in I}\sum_{p=1}^m(-1)^{|i||j|+|i||k|}\theta_k\theta_j\Gamma_{ik}^{(0,p-1)}
\widetilde{\Gamma}_{\bar{j}\bar{k}}^{(0,m-p)}
+(-1)^{|i||j|+|i|}\theta_i\theta_j\widetilde{\Gamma}_{\bar{j}\bar{i}}^{(0,m)}+\Gamma_{ij}^{(0,m)}
\end{align*}
modulo $\mathsf{F}_{[m+1]}$.
\end{lemma}

\begin{theorem}\label{embedding}
For $m\geqslant 0$, we have $\varphi\left(\hat{s}_{ij}^{(m+1)}\right)=\overline{\Lambda_{ij}^{(1,m)}} \in\mathsf{F}_{[m]}/\mathsf{F}_{[m+1]}$, where $\varphi$ is the isomorphism mentioned in Theorem \ref{main1}.
\end{theorem}

\begin{proof}
By relation \eqref{Ytw4} in $\mathrm{Y}_{\hbar}\big(\mathfrak{gl}_{M|2n}\big)$ and Lemma \ref{LM3}, we have $\varphi\left(\hat{s}_{i,j}^{(m+1)}\right)\in \mathsf{F}_{[m]}$ and
\begin{align*}
\varphi\left(\hat{s}_{ij}^{(m+1)}\right)&=\hbar\sum_{k\in I}\sum_{p=1}^m(-1)^{p}(-1)^{|i||j|+|j||k|}\theta_k\theta_j
\varphi\left(\hat{t}_{ik}^{(m+1-p)}\right)\varphi\left(\hat{t}_{\bar{j}\bar{k}}^{(p)}\right) \\
&\hspace{5em}+(-1)^{m+1}\theta_i\theta_j\varphi\left(\hat{t}_{\bar{j}\bar{i}}^{(m+1)}\right)
+(-1)^{|i||j|+|j|}\varphi\left(\hat{t}_{ij}^{(m+1)}\right) \\
&\equiv\left(q-q^{-1}\right)\sum_{k\in I}\sum_{p=1}^m (-1)^{p}(-1)^{|i||j|+|i||k|}\theta_k\theta_j \Gamma_{ik}^{(0,m-p)}\Gamma_{\bar{j}\bar{k}}^{(0,p-1)} \\
&\hspace{5em}+(-1)^{m+1} (-1)^{|i||j|+|i|}\theta_i\theta_j\Gamma_{\bar{j}\bar{i}}^{(0,m)}
+\Gamma_{ij}^{(0,m)} \\
&\equiv \Lambda_{ij}^{(1,m)}
\end{align*}
modulo $\mathsf{F}_{[m+1]}$.

\end{proof}

Let $\mathsf{F}^{tw}_{[m]}$ be the span of the form
\begin{gather*}
\left(q-q^{-1}\right)^{m_0}\Lambda_{i_1,j_1}^{(1,m_1)}\Lambda_{i_2,j_2}^{(1,m_2)}\ldots
\Lambda_{i_k,j_k}^{(1,m_k)},
\end{gather*}
where $m_0+\cdots+m_k\geqslant m$. Denote $\overline{V}^{tw}=\bigoplus_{m=0}^{\infty}\mathsf{F}^{tw}_{[m]}/\mathsf{F}^{tw}_{[m+1]}$.
Then $\Lambda_{ij}^{(r,m)}\equiv\Lambda_{ij}^{(1,m)}$ still holds if modulo $\mathsf{F}^{tw}_{[m+1]}$ which follows from $\Lambda_{ij}^{(r,m)}=\sum_{i=1}^{r-1}\binom{r-1}{i}\Lambda_{ij}^{(1,m+i)}+\Lambda_{ij}^{(1,m)}\in \mathsf{F}^{tw}_{[m]}$.
Let $\varphi^{tw}$ be the morphism that maps $\hat{s}_{ij}^{(m+1)}$ into $\overline{\Lambda_{ij}^{(1,m)}}\in\mathsf{F}^{tw}_{[m]}/\mathsf{F}^{tw}_{[m+1]}$. In $\overline{V}^{tw}$, set $\hbar=\overline{q-q^{-1}}\in\mathsf{F}^{tw}_{[1]}/\mathsf{F}^{tw}_{[2]}$. We can identify $\varphi^{tw}$ with the restriction of $\varphi$ on $\mathrm{Y}^{tw}_{\hbar}\big(\mathfrak{osp}_{M|2n}\big)$.

Let $\mu_1$ be the embedding of $\mathrm{Y}^{tw}_{\hbar}\big(\mathfrak{osp}_{M|2n}\big)$ into $\mathrm{Y}_{\hbar}\big(\mathfrak{gl}_{M|2n}\big)$ defined by relation \eqref{Ytw4}. The embedding $\mu_2$ can be defined by Lemma \ref{LM3}  similarly. As a consequence, we have the following commutative diagram
\[\xymatrix@C=80pt{
\mathrm{Y}^{tw}_{\hbar}\big(\mathfrak{osp}_{M|2n}\big) \ar[d]_{\varphi^{tw}  } \ar[r]^{\mu_1} &   \mathrm{Y}_{\hbar}\big(\mathfrak{gl}_{M|2n}\big) \ar[d]^{\varphi}     \\
\overline{V}^{tw}  \ar[r]^{\mu_2} & gr\mathrm{U}_{\mathbb{A}},    } \]
that is, $\mu_2\cdot\varphi^{tw}=\varphi\cdot\mu_1$.
Therefore we have the main result as follows.

\begin{theorem}\label{main2}
The morphism $\varphi^{tw}$ is a superalgebraic isomorphism between $\mathrm{Y}^{tw}_{\hbar}(\mathfrak{osp}_{M|2n})$ and the graded algebra $\overline{V}^{tw}$.
\end{theorem}

\section*{Acknowledgments}
The authors express their gratitude to the reviewers for valuable suggestions. Special thanks are also extended to the visiting professor, Alexander Molev, for numerous insightful discussions and for recommending the reference \cite{GoMo}. Y. Wang is partially supported by the National Natural Science Foundation of China (No. 12071026), the Fundamental Research Funds for the Central Universities JZ2021HGTB0124, and the Anhui Provincial Natural Science Foundation 2308085MA01. H. Zhang is partially supported by the support of the National Natural Science Foundation of China (No. 12271332), and Natural Science Foundation of Shanghai 22ZR1424600.

\end{document}